\documentclass[leqno]{article}

\usepackage{amsmath,amssymb,amsthm,dsfont,bm,mathrsfs,comment,yhmath,authblk,xcolor,csquotes}

\usepackage[hidelinks]{hyperref}

\usepackage[a4paper,top=3cm,bottom=2.5cm,left=2cm,right=2cm]{geometry}

\newtheorem{theorem}{Theorem}[section]
\newtheorem{proposition}[theorem]{Proposition}
\newtheorem{lemma}[theorem]{Lemma}
\newtheorem{corollary}[theorem]{Corollary}

\theoremstyle{definition}
\newtheorem{remark}[theorem]{Remark}

\numberwithin{equation}{section}

\newcommand{\N}{\mathbb{N}}
\newcommand{\Z}{\mathbb{Z}}

\newcommand{\R}{\mathbb{R}}

\newcommand{\cont}{\mathcal{C}}
\newcommand{\sch}{\mathcal{S}}
\newcommand{\ld}{{L^2}}
\newcommand{\lp}{{L^p}}
\newcommand{\li}{{L^\infty}}
\newcommand{\hs}{{H^s}}

\newcommand{\supp}{\mathrm{supp}\,}

\newcommand{\weak}{\rightharpoonup}
\newcommand{\weaks}{\overset{\ast}{\rightharpoonup}}
\newcommand{\into}{\hookrightarrow}
\newcommand{\impl}{\Longrightarrow}

\newcommand{\pt}{\partial}
\newcommand{\dt}{\partial_t}

\newcommand{\nb}{\nabla}
\newcommand{\dd}{{\rm\,d}}
\newcommand{\ddt}{\frac{\dd}{\dd t}}
\newcommand\divv{\mathrm{div}}
\newcommand\curl{\mathrm{curl}}

\newcommand{\T}{\mathcal{T}}
\newcommand{\RR}{\mathcal{R}}

\newcommand{\id}{I}

\newcommand{\al}{\alpha}
\newcommand{\om}{\omega}
\newcommand{\Om}{\Omega}

\newcommand{\A}{\mathcal{A}}
\newcommand{\B}{\mathcal{B}}
\newcommand{\D}{\mathcal{D}}
\newcommand{\E}{\mathbb{E}}

\newcommand{\e}{\mathcal{E}}
\newcommand{\f}{\mathcal{F}}
\newcommand{\U}{\mathbb{U}}

\newcommand{\hsm}{{H^{s-1}}}
\newcommand{\hsmd}{{H^{s-2}}}
\newcommand{\hsp}{{H^{s+1}}}
\newcommand{\hspd}{{H^{s+2}}}

\newcommand{\hsig}{{H^\sigma}}

\newcommand{\del}{\Delta}

\newcommand{\eps}{\varepsilon}
\newcommand{\phii}{\varphi}
\newcommand{\lut}{L_T^1}
\newcommand{\ldt}{L_T^2}

\newcommand{\lit}{L_T^\infty}
\newcommand{\lutt}{\widetilde{L_T^1}}

\newcommand{\litt}{\widetilde{L_T^\infty}}
\newcommand{\ct}{C_T}
\newcommand{\cut}{C_T^1}
\newcommand{\ctt}{\widetilde{C_T}}

\newcommand{\sig}{\sigma}

\newcommand{\uz}{u_0}
\newcommand{\rz}{\rho_0}
\newcommand{\rez}{\rho_0}

\newcommand{\uzeps}{u_0}
\newcommand{\vzeps}{U_0}

\newcommand{\re}{\rho_\eps}
\newcommand{\aeps}{a_\eps}
\newcommand{\ue}{u_\eps}

\newcommand{\teeps}{T_\eps}
\newcommand{\teepse}{T_{\eps,*}}

\newcommand{\ve}{U_\eps}
\newcommand{\Pe}{\Pi_\eps}
\newcommand{\eeps}{\E_\eps}

\newcommand{\ps}{[s]}

\newcommand{\nbp}{\nabla^{\bot}}
\newcommand{\up}{u^{\bot}}
\newcommand{\tns}{\tau^\mathrm{NS}}
\newcommand{\todd}{\tau^\mathrm{odd}}
\newcommand{\muodd}{\mu_0}

\title{Well-posedness for 2D non-homogeneous incompressible fluids with general density-dependent odd viscosity}

\author{Matthieu Pageard\thanks{Université Claude Bernard Lyon 1, CNRS, Centrale Lyon, INSA Lyon, Université Jean Monnet, ICJ UMR5208,
69622 Villeurbanne, France. Email address: \url{pageard@math.univ-lyon1.fr}.}}

\date{\today}

\begin{document}

\maketitle

\abstract{We study the initial value problem for a system of equations describing the motion of two-dimensional non-homogeneous incompressible fluids exhibiting odd (non-dissipative) viscosity effects. We consider the complete odd viscous stress tensor with a general density-dependent viscosity coefficient $f(\rho)$. Under suitable assumptions, we prove the local existence and uniqueness of strong solutions in $\hs(\R^2)$ $(s>2)$, for a class of viscosity coefficients covering the particular case $f(\rho)=a\rho^\alpha+b$ for any $(a,b,\alpha)\in\R^3$, generalising the result of \cite{fgs} devoted to the case $f(\rho)=\rho$. Additionally, we are able to do so without requiring the initial density variation to belong to $\ld(\R^2)$. As a major step of the proof, we exhibit an effective velocity for this sytem, generalising the so-called ``Elsässer formulation'' recently derived in \cite{fv}.}



\tableofcontents

\section{Introduction}
\label{sec:intro}

When a fluid is set into motion, different parts of the fluid may move at different speeds. Mathematically, this is characterised by the appearance of inhomogeneities in its velocity field, or \emph{velocity gradients}. Viscosity quantifies a fluid’s ability to resist deformation due to velocity gradients, by exerting internal stresses. At each point of the fluid, the local average of these stresses is called \emph{viscous stress}.  When the viscous stresses are proportional to the velocity gradients, the fluid is said to be \emph{Newtonian}, and the proportionality coefficient is called \emph{viscosity}. This relationship takes the form\footnote{We adopt the Einstein convention of summation over repeated indices.} \cite{batchelor1967introduction}
\begin{equation}
    \label{tau}
    \tau_{ij} = \eta_{ijkl}\,\pt_l u_k,
\end{equation}
where $\tau_{ij}$ is the \emph{viscous stress tensor}, $\eta_{ijkl}$ is the \emph{viscosity tensor}, and $u$ is the velocity field. At the microscopic level, the viscous stress manifests as the action of intermolecular forces, converting momentum into disordered molecular motion, or heat. We refer to \cite[Chapter 1]{batchelor1967introduction} for further insight on this phenomenon.

\medskip
This description, commonly used to define viscosity, naturally relates it to the dissipation of kinetic energy. However, this connection is not universal, as only the symmetric part of the viscosity tensor contributes to energy dissipation. In order to show this, let us introduce the compressible Navier--Stokes equations in $\R^d$ ($d\ge1$)\footnote{We define the divergence of a matrix $A$ as the vector $(\divv\,A)_i = \sum_{j}\pt_jA_{ji}$.}:
\begin{equation}
    \label{ns}
    \left\{
    \begin{aligned}
    & \partial_t \rho + \divv(\rho u) = 0, \\
    & \partial_t(\rho u) + \divv(\rho u\otimes u) = -\nb\pi + \divv\,\tau,
\end{aligned}
\right.
\end{equation}
where $\rho=\rho(t,x)\ge0$ is the density, $u=u(t,x)\in\R^d$ is the velocity field, $\pi=\pi(t,x)\in\R$ is the pressure, and $\tau=\tau_{ij}(t,x)\in\mathrm{Mat}_{d\times d}(\R)$ is the viscous stress tensor. 

\medskip
The total kinetic energy of a fluid governed by equations \eqref{ns} is\footnote{For a function $f=f(t,x)$ with $x\in\R^d$, we denote $\int f \dd x = \int_{\R^d} f(t,x) \dd x$.} \cite{landau1987fluid}
\[
E_\mathrm{kin}(t) := \frac{1}{2}\int \rho|u|^2\dd x.
\]
Elementary computations show that the total variation of kinetic energy is\footnote{We denote by $A:B$ the canonical inner product of two matrices $A$ and $B$.}
\[
\dot E_\mathrm{kin}(t) = - \int \nb\pi\cdot u\dd x - \int \tau:\nb u \dd x.
\]
From expression \eqref{tau}, we see that, after decomposing $\eta=\eta^\mathrm{S}+\eta^\mathrm{A}$ into its symmetric and antisymmetric components
\[
\eta_{ijkl}^\mathrm{S} = \frac{1}{2}\big(\eta_{ijkl}+\eta_{klij}\big),\qquad \eta_{ijkl}^\mathrm{A} = \frac{1}{2}\big(\eta_{ijkl}-\eta_{klij}\big),
\]
the contribution of the viscous stress tensor is 
\[
\int \tau:\nb u\dd x = \int \eta_{ijkl}^\mathrm{S}(\pt_lu_k)(\pt_iu_j)\dd x.
\]
Therefore, the antisymmetric part $\eta^\mathrm{A}$ of the viscosity tensor does not contribute to the total variation of kinetic energy, and corresponds to a transverse, \emph{non-dissipative} stress. It is often referred to as \emph{odd viscosity} (because $\eta_{klij}^\mathrm{A}=-\eta_{ijkl}^\mathrm{A}$), or \emph{Hall viscosity} in the context of condensed matter physics, \cite{avron1995viscosity}.

\medskip
Odd viscosity usually arises in physical systems where the microscopic dynamics do not obey time-reversal symmetry. Such symmetry breaking may occur, for instance, due to self-rotating particles, or in the presence of an external magnetic field, \cite{fruchart2023odd}. Odd viscosity can produce counterintuitive effects, as illustrated by the following example, see \emph{e.g.} \cite{avron1998odd,LapaHughes2014}. A rotating cylinder immersed in a fluid with classical dissipative viscosity will encounter a resistance opposite to the direction of rotation, whereas, when immersed in a fluid with odd viscosity, it will undergo a \emph{radial} pressure, proportional to the rate of rotation, and oriented inwards or outwards, depending on the direction of rotation. Systems exhibiting odd viscosity effects are, \emph{e.g.}, polyatomic gases under a magnetic field \cite{polyatomic}, magnetised plasmas \cite{plasma}, gases under rotation \cite{gasrotation}, superfluids \cite{superfluid}, chiral active fluids \cite{chiralactive}, or vortex matter \cite{vortex}. We refer to \cite{fruchart2023odd} for an overview of this subject. Planar fluids are of particular interest, being the only situation where odd viscosity is compatible with isotropy, \cite{avron1998odd}.

\paragraph{The viscous stress tensor in dimension 2.} In an isotropic two-dimensional fluid, the viscosity tensor $\eta$ has 6 independent coefficients, and takes the form\footnote{In this expression, we use the volume viscosity $\lambda$ instead of the bulk viscosity $\zeta$. This results in the change $\zeta=\lambda+\mu$ in the expression (10) of \cite{fruchart2023odd}.} \cite{fruchart2023odd}
\begin{equation}
    \label{eta:2d}
    \begin{split}
        \eta_{ijkl} = \,\mu\,&\big(\delta_{ik}\delta_{jl}+\delta_{il}\delta_{jk}\big) + \lambda\,\delta_{ij}\delta_{kl} + \mu_\mathrm{rot}\,\epsilon_{ij}\epsilon_{kl} \\
        +\,\muodd\,&\big(\epsilon_{ik}\delta_{jl}+\epsilon_{jl}\delta_{ik}\big) - \mu_1\,\epsilon_{ij}\delta_{kl} -\mu_2\,\delta_{ij}\epsilon_{kl},
    \end{split}
\end{equation}
where $\epsilon$ denotes the Levi-Civita symbol, whose coefficients are given by the matrix 
\[
J 
=
\begin{pmatrix}
0 & 1 \\
-1 & 0 
\end{pmatrix}.
\]
The pairs $(\mu,\muodd)$, $(\lambda,\mu_1)$ and $(\mu_\mathrm{rot},\mu_2)$ are related respectively to shear, volumetric and rotational stresses, while the triplets $(\mu,\lambda,\mu_\mathrm{rot})$ and $(\muodd,\mu_1,\mu_2)$ correspond respectively to dissipative and odd viscosities. These coefficients are not constant in principle, and usually depend on the state variables\footnote{The \emph{state variables} entirely describe the state of a given mass of fluid at equilibrium. All the other quantities are functions of the state variables through the \emph{equations of state}, \cite{batchelor1967introduction}.} of the system, namely temperature and density, \cite{batchelor1967introduction}.

\medskip
Let us notice that the viscosity tensor corresponding to the standard compressible Navier--Stokes equations, namely
\[
\eta_{ijkl}^\mathrm{NS} = \mu\,\big(\delta_{ik}\delta_{jl}+\delta_{il}\delta_{jk}\big) + \lambda\,\delta_{ij}\delta_{kl}
\]
is symmetric, as it satisfies $\eta_{klij}^\mathrm{NS}=\eta_{ijkl}^\mathrm{NS}$. The viscosity of such flows is therefore of fully dissipative nature. 

\medskip
Plugging the expression \eqref{eta:2d} of the viscosity tensor into \eqref{tau}, we can write the corresponding viscous stress tensor 
\begin{equation}
    \label{tau:2d}
    \begin{split}
        \tau = \,\mu&\big(\nb u + (\nb u)^\mathrm{T}\big) + \lambda(\divv\,u)I - \mu_\mathrm{rot}\,\om J \\
        -\,\muodd&\big(\nb\up + \nbp u\big) - \mu_1(\divv\,u)J + \mu_2\,\om I,
    \end{split}
\end{equation}
where $I$ is the identity matrix, $\om=\curl(u)=\pt_1u_2 - \pt_2u_1$ is the vorticity, $\up=(-u_2,u_1)$, and $\nbp=(-\pt_2,\pt_1)^\mathrm{T}$. This expression is the most general form of the viscous stress tensor for an isotropic planar fluid.

\paragraph{The system of equations.} In this work, we consider \emph{incompressible} fluids governed by equations \eqref{ns}, with the viscous stress tensor
\[
\tau^\mathrm{odd} := -\mu_0\big(\nb\up + \nbp u\big).
\]
The viscosity coefficient $\muodd$ will be assumed to depend only on density, as \emph{e.g.} in \cite{hamiltonian}: we set
\[
\mu_0 = f(\rho),
\]
where $f:\R^+\to\R$ is a given function. 

\medskip
The full system of equations in $\R^+\times\R^2$, given by \eqref{ns} with $\tau=\todd$, and supplemented with the divergence-free condition on $u$, reads
\begin{equation}
\label{odd}
\left\{
\begin{aligned}
& (\dt + u\cdot\nb)\rho = 0, \\
& \rho(\dt+u\cdot\nb)u + \nabla\pi + \divv\big(f(\rho)(\nabla u^{\bot} + \nabla^{\bot}u)\big) = 0, \\
& \divv\,u = 0,
\end{aligned}
\right.
\end{equation}
where $\rho=\rho(t,x)\ge0$, $u=u(t,x)\in\R^2$, and $\pi=\pi(t,x)\in\R$.

\paragraph{Notation.} For a vector 
\[
u=(u_1,u_2)\in\R^2,
\]
we define its gradient as the matrix
\[
\nb u:=
\left(\begin{array}{@{}c|c@{}}
\nb u_1 & \nb u_2
\end{array}\right)
=
\begin{pmatrix}
\pt_1 u_1 & \pt_1 u_2 \\
\pt_2 u_1 & \pt_2 u_2 
\end{pmatrix}.
\]
We introduce the rotations 
\[
\up:=(-u_2,u_1),\qquad 
\nbp := 
\begin{pmatrix}
-\pt_2 \\
\pt_1 
\end{pmatrix},
\]
and define
\[
\nb\up := 
\left(\begin{array}{@{}c|c@{}}
-\nb u_2 & \nb u_1
\end{array}\right)
=
\begin{pmatrix}
-\pt_1 u_2 & \pt_1 u_1 \\
-\pt_2 u_2 & \pt_2 u_1 
\end{pmatrix},\qquad
\nbp u := 
\left(\begin{array}{@{}c|c@{}}
\nbp u_1 & \nbp u_2
\end{array}\right)
=
\begin{pmatrix}
-\pt_2 u_1 & -\pt_2 u_2 \\
\pt_1 u_1 & \pt_1 u_2 
\end{pmatrix}.
\]

\subsection{Mathematical structure and previous results} 

We are interested in the initial value problem for system \eqref{odd}. The main challenge in studying this system lies in the fact that the viscous stress tensor $\todd$ is non-dissipative: one has 
\[
\int \todd:\nb u \dd x = \int f(\rho) \big(\nb\up:\nb u\big) \dd x + \int f(\rho) \big(\nbp u:\nb u\big) \dd x = 0.
\]
Hence, in contrast to the standard Navier--Stokes equations, no gain of regularity can be expected from it. On the other hand, we may decompose the odd viscosity term as 
\begin{equation}
    \label{oddterm}
    \divv\big(f(\rho)(\nabla u^{\bot} + \nabla^{\bot}u)\big) = \big(\nb f(\rho)\cdot\nb\big)\up + f(\rho)\del\up + \big(\nb f(\rho)\cdot\nbp\big)u.
\end{equation}
From this expression, and assuming that the density is bounded away from vacuum ($\rho\ge\rho_*>0$), one can rewrite the second equation in \eqref{odd} as the transport equation 
\[
(\pt_t+u\cdot\nb)u = -\frac{1}{\rho}\nb\pi - \frac{1}{\rho}\big(\nb f(\rho)\cdot\nb\big)\up - \frac{f(\rho)}{\rho}\del\up - \frac{1}{\rho}\big(\nb f(\rho)\cdot\nbp\big)u.
\]
Because the odd viscosity term contains two derivatives in $u$ and one derivative in $\rho$, the classical results from transport theory are not directly applicable, as one cannot merely consider it as a forcing term. In the following, we present previous results related to the mathematical study of fluids with odd viscosity, as well as the methods developed to handle the loss of derivatives caused by the odd viscosity term.

\paragraph{Previous results.} To the best of our knowledge, the first mathematical work related to fluids with odd viscosity is \cite{GraneroBelinchonOrtega2022}. The first study on the initial value problem for system \eqref{odd} is \cite{fgs}, where the authors discarded the term $\nbp u$. Moreover, they supposed $f(\rho)=\rho$. One may comment more on the importance of this assumption in \cite{fgs}. Obviously, even with these simplifications, the loss of derivatives mentioned above is still present. However, in this case, an underlying hyperbolic structure appears in system \eqref{odd}, allowing to overcome this issue. More precisely, the identification of the quantities
\[
\om := \curl(u) = \pt_1u_2-\pt_2u_1,\qquad \theta := \curl(\rho u) - \del\rho
\]
leads to the system of transport equations 
\[
\left\{
\begin{aligned}
& \big(\dt + (u-\nbp\log\rho)\cdot\nb\big)\om  = -\nbp\left(\frac{1}{\rho}\right)\cdot\nb(\pi-\rho\om) - \B(\nb u,\nb^2\log\rho), \\
& (\dt + u\cdot\nb)\theta = \frac{1}{2}\nbp\rho\cdot\nb|u|^2 + \B(\nb u,\nb^2\rho).
\end{aligned}
\right.
\]
Starting from initial data $\rho_0-1\in\hsp(\R^2)$ and $u_0\in\hs(\R^2)$ ($s>2$), it was proved that the modified pressure gradient $\nb(\pi-\rho\om)$ and the bilinear terms $\B$ belong to $\hsm(\R^2)$, while the new transport field for the vorticity $\om$, namely $u-\nbp\log\rho$, belongs to $\hs(\R^2)$ and is divergence-free. From these observations, the authors were able to gather $\hsm$ regularity for the new variables $\om$ and $\theta$, and therefore the desired regularity $\rho-1\in\hsp(\R^2)$ and $u\in\hs(\R^2)$ for the density and velocity fields.

\medskip
The first attempt to tackle the full odd viscosity system \eqref{odd} was made in \cite{fv}, still in the case $f(\rho)=\rho$. This result arises from the fundamental identity 
\begin{equation}
    \label{gradients}
    \nb\up - \nbp u = -\om\id
\end{equation}
which allows to write 
\[
\nb\up+\nbp u = 2\nb\up + \om\id.
\]
Applying the divergence operator, the second term on the right is then a gradient, that can be absorbed in the pressure term, leading to the same quantity $\nb(\pi-\rho\om)$ appearing in the approach of \cite{fgs}. One is then reconducted to the system studied in \cite{fgs}, and a similar well-posedness theory naturally follows. The initial value problem is then investigated in endpoint Besov spaces $B_{\infty,r}^s(\R^2)$ ($s>1$ or $s=r=1$). The crucial point is the identification of an \emph{effective velocity}, which may be written as 
\[
W_\mathrm{eff} := u - 2\nbp\log\rho,
\]
and is also divergence-free. This kind of quantity, linking the velocity field with the gradient of some function of the density, was already identified in compressible fluid mechanics, \emph{e.g.} by Lions \cite{lions1998comp} in a low Mach number model, and by Bresch-Desjardins \cite{bd2003,bd2004,bd2006} for the shallow-water (or Saint-Venant) equations. The two velocities $u$ and $W_\mathrm{eff}$ are transported by each other, giving rise to the following reformulation of the odd viscosity system \eqref{odd}:
\[
\left\{
\begin{aligned}
& (\pt_t + u\cdot\nb)\rho = 0, \\
& \rho(\pt_t + W_\mathrm{eff}\cdot\nb)u + \nb\Pi^0 = 0, \\
& \rho(\pt_t + u\cdot\nb)W_\mathrm{eff} + \nb\Pi^0 = 0, \\
& \divv\,u = \divv\,W_\mathrm{eff} = 0,
\end{aligned}
\right.
\]
where $\Pi^0 := \pi - \rho\om$. As already observed in \cite{fv}, this new system is similar to the Elsässer formulation of the ideal MHD equations, see \emph{e.g.} \cite{cobb2023elsasser} and references therein. Notice that the quantities $W_\mathrm{eff}$ and $\Pi^0$ already appeared in the approach of \cite{fgs}, but only at the level of the equations for $\om$ and $\theta$. Let us remark that this reformulation allows to obtain a greater regularity for the modified pressure: while in \cite{fgs}, we had only $\nb\Pi^0\in\hsm(\R^2)$, we have this time $\nb\Pi^0\in B_{\infty,r}^s(\R^2)$. This gain of one derivative will be crucial in our approach, as it will allow us to work directly with the analogue of the variables $u$ and $W_\mathrm{eff}$ rather than their vorticities, making the analysis much more straightforward.

\medskip
The present paper aims to extend the result of \cite{fgs} to the case of the full odd viscosity tensor with a general density-dependent viscosity coefficient. 

The analysis of fluid models with density-dependent viscosity coefficients sparked a great interest in the mathematical community over the last 30 years. Let us give a (far from complete) overview of the results obtained regarding the \emph{compressible} Navier--Stokes equations, given by system \eqref{ns} with 
\begin{equation}
    \label{nstensor}
    \tns(\rho) = \mu(\rho)\big(\nb u + (\nb u)^\mathrm{T}\big) + \lambda(\rho)(\divv\,u)I.
\end{equation}
There are two main classes of results.

The results related to the existence of strong solutions are mainly restricted to low dimensions $d=1,2$. The seminal work of Vaigant and Kazhikhov \cite{vaigant1995} establishes the global existence of classical solutions in some bounded domain $\Om\subset\R^2$, when $\mu$ is constant and $\lambda$ is a suitable power of the density. We also mention the work of Mellet and Vasseur \cite{mv2008} for global strong solutions in dimension $d=1$. Some results in higher dimensions also exist, see \emph{e.g.} Danchin \cite{danchinlagrangian2012} for local strong solutions in dimension $d\ge2$.

The literature for weak solutions is much richer. Lions \cite{lions1998comp} and Fereisl \cite{feireisl2004dynamics} prove global existence of weak solutions when the viscosity is constant. The pioneering works of Bresch-Desjardins-Lin \cite{bdl2003} and Bresch-Desjardins \cite{bd2003,bd2004,bd2006} for the viscous shallow-water equations and the compressible Navier--Stokes equations paved the way for tackling the case of density-dependent viscosities, with the introduction of the so-called BD-entropy. These results were later complemented by Mellet-Vasseur \cite{mv2007}, who derived a new logarithmic entropy inequality. These works allow for a real breakthrough with the existence of global weak solutions in dimension $d=3$, proved independently by Li-Xin \cite{lixin2015} and Vasseur-Yu \cite{vy2016existence}, later extended by Bresch-Vasseur-Yu \cite{bvy2021} to a more general viscous stress tensor.

Such results also exist for the non-homogeneous incompressible Navier--Stokes equations, given by \eqref{ns}-\eqref{nstensor} with $\divv\,u=0$. The first results regarding the existence of global weak solutions in the density-dependent case are due to Lions \cite{Lions1996incomp} and Desjardins \cite{Desjardins1997}. As for strong solutions, more results are available. The work of Danchin \cite{danchin2003density} for a constant viscosity was generalised by Abidi \cite{Abidi2007_Critique} and Abidi-Paicu \cite{AbidiPaicu2007} to the case of a density-dependent viscosity, see also Abidi-Gui-Zhang \cite{AbidiGuiZhang2012}, Huang-Wang \cite{HuangWang2015}, and Zhang \cite{Zhang2015_JDE}. We also mention some recent works on the case of a density jump, see \emph{e.g.} \cite{dm2019,dm2023,gancedo20252dnavierstokesfreeboundary} and references therein. 

\medskip
As the mathematical analysis of fluid models with odd viscosity is still very new, the only result dealing with general density-dependent viscosity coefficients in this case is the recent paper by Zimmermann \cite{zimmermann}, which investigates non-homogeneous incompressible fluids governed by system \eqref{ns} (with $\divv\,u=0$) in some domains $\Om\subset\R^2$, with a stress tensor $\tau$ displaying both the classical dissipative viscosity, as well as the odd viscosity, namely
\[
\tau^\mathrm{NS+odd}(\rho) = \mu(\rho)\big(\nb u + (\nb u)^\mathrm{T}\big) - \muodd(\rho)\big(\nb\up + \nbp u\big).
\]
For this system, the author establishes the existence of weak solutions in both the evolutionary and stationary cases.

\medskip
In this work, we extend the result of \cite{fgs} to fluids with general density-dependent odd viscosity by proving local existence and uniqueness of strong solutions to system \eqref{odd}. Before stating our result, we need to reformulate system \eqref{odd} as done in \cite{fv}.

\subsection{Reformulation of the equations}
\label{sec:reformulation}

The computations performed in this section are for the moment only formal. The goal is to rewrite \eqref{odd} as a system of transport equations. The fundamental tool for this purpose is the identity \eqref{gradients}, pointed out in \cite{fv}, allowing to decompose the odd viscosity term as 
\[
\begin{split}
     \divv\big(f(\rho)(\nabla u^{\bot} + \nabla^{\bot}u)\big) &= \divv\big(f(\rho)(-\om\id + 2\nbp u)\big) \\
     &= -\nb\big(f(\rho)\om\big) + 2\divv\big(f(\rho)\nbp u\big).
\end{split}
\]
After defining the modified pressure field
\begin{equation}
    \label{Pi}
    \Pi:=\pi-f(\rho)\om,
\end{equation}
the second equation in \eqref{odd} becomes 
\begin{equation}
    \label{moment2}
    \rho(\dt + u\cdot\nb)u + \nb\Pi + 2\divv\big(f(\rho)\nbp u\big) = 0,
\end{equation}
The new simplified divergence term above is in fact a transport term, as
\[
\divv\big(f(\rho)\nbp u\big) = (\nb f(\rho)\cdot\nbp)u = -(\nbp f(\rho)\cdot\nb)u = -f'(\rho)(\nbp\rho\cdot\nb)u.
\]
Let us assume that the density $\rho$ is positively bounded from below ($\rho\ge\rho_*>0$), and consider a function $g$ such that 
\[
g'(\rho) = \frac{2}{\rho}f'(\rho).
\]
We then have 
\[
2\divv\big(f(\rho)\nbp u\big) = -2f'(\rho)(\nbp\rho\cdot\nb)u = -\rho g'(\rho)(\nbp\rho\cdot\nb)u = -\rho(\nbp g(\rho)\cdot\nb)u.   
\]
We now define the \emph{effective velocity}
\begin{equation}
    \label{U}
    U := u - \nbp g(\rho),
\end{equation}
which is also divergence-free. The odd viscosity system \eqref{odd} thus rewrites
\begin{equation}
\label{odd3}
\left\{
\begin{aligned}
& (\dt+u\cdot\nb)\rho = 0, \\
& \rho(\dt + U\cdot\nb)u + \nabla\Pi = 0, \\
& \divv\,u = \divv\,U = 0.
\end{aligned}
\right.
\end{equation}
Therefore, the original velocity field $u$ is transported by the new effective velocity $U$. At this stage, we have drastically improved the problem of the loss of derivatives mentioned earlier. Indeed, as \eqref{odd3} is an Euler-type equation with a modified transport field, we now only lose one derivative in the variable $\rho$, the other losses having been absorbed in both $U$ and $\nb\Pi$. In order to establish the well-posedness of system \eqref{odd3}, we now thus only need to have $U\in\hs$. 

Here comes into play the ``Elsässer formulation'' derived in \cite{fv}. It turns out that the new effective velocity $U$ itself is transported by $u$, which will allow us to obtain the desired regularity on $U$. Let us now compute the Elsässer formulation of system \eqref{odd} in our context.

\medskip
From the expression \eqref{U} of the effective velocity, we can decompose
\[
\begin{split}
    (U\cdot\nb)u &= (u\cdot\nb)u - (\nbp g(\rho)\cdot\nb)u \\
    &= (u\cdot\nb)U + (u\cdot\nb)\nbp g(\rho) + (\nb g(\rho)\cdot\nbp)u.
\end{split}
\]
The second equation in \eqref{odd3} then rewrites
\[
\rho(\dt + u\cdot\nb)U + \nb\Pi + \rho\Big(\pt_t\nbp g(\rho) + (u\cdot\nb)\nbp g(\rho) + (\nb g(\rho)\cdot\nbp)u\Big) = 0.
\]
Writing that 
\[
(u\cdot\nb)\nbp g(\rho) = \nbp(u\cdot\nb)g(\rho) - (\nb g(\rho)\cdot\nbp)u,
\]
we obtain the equation 
\[
\rho(\dt + u\cdot\nb)U + \nb\Pi + \rho\nbp(\pt_t+u\cdot\nb)g(\rho) = 0.
\]
From the first equation in \eqref{odd3}, we have
\[
(\dt + u\cdot\nb)g(\rho) = g'(\rho)(\dt + u\cdot\nb)\rho = 0,
\]
which yields the desired transport equation
\[
\rho(\dt+ u\cdot\nb)U + \nb\Pi = 0.
\]
Finally, as 
\[
\nbp g(\rho)\cdot\nb\rho = g'(\rho)\nbp\rho\cdot\nb\rho = 0,
\]
we have 
\[
(\dt + U\cdot\nb)\rho = 0.
\]
We thus obtain the \emph{Elsässer formulation} of the odd viscosity system \eqref{odd}, namely 
\begin{equation}
    \label{els}
    \left\{
    \begin{aligned}
        & (\pt_t + u\cdot\nb)\rho = 0, \\
        & (\pt_t + U\cdot\nb)\rho = 0, \\
        & \rho(\pt_t + U\cdot\nb)u + \nb\Pi = 0, \\
        & \rho(\pt_t + u\cdot\nb)U + \nb\Pi = 0, \\
        & \divv\,u = \divv\,U = 0.
    \end{aligned}
\right.
\end{equation}
This formulation was already derived in a slightly more complex way in \cite{fv} in the particular case $f(\rho)=\rho$. 

\medskip
Let us point out that this system is not equivalent to the original system \eqref{odd} in the following sense. Suppose that we have a solution $(\rho,u,\nb\pi)$ to system \eqref{odd}. Then, one can formally reproduce the computations above to find that the quadruple $(\rho,u,U,\nb\Pi)$, where $\Pi$ and $U$ are defined by \eqref{Pi} and \eqref{U}, solves the Elsässer formulation \eqref{els}. However, the opposite is a priori not true. Indeed, suppose this time that we have a quadruple $(\rho,u,U,\nb\Pi)$ solution to system \eqref{els}. Obviously, one can define 
\[
\pi := \Pi + f(\rho)\om,
\]
but to recover that the triplet $(\rho,u,\nb\pi)$ satisfies \eqref{odd}, we also need the relation 
\[
U = u - \nbp g(\rho)
\]
to be satisfied, which is a priori not the case. As a consequence, one cannot solve the Elsässer formulation \eqref{els} in the variables $(\rho,u,U,\nb\Pi)$ to obtain a solution to system \eqref{odd}, as was done \emph{e.g.} in \cite{cobb2023elsasser} for the ideal MHD equations.

\subsection{Main result} 

The main result of this paper is the following\footnote{We denote by $\ps$ the lower integer part of $s$, namely $\ps:=\max\{n\in\N:n\le s\}$.}.

\begin{theorem}
    \label{main}
    Let $s>2$. Let $(\rho_0,u_0)\in\li(\R^2)\times\hs(\R^2)$ be such that  
    \[
    0<\rho_*:=\inf_{x\in\R^2}\rho_0(x)\le\rho^*:=\|\rho_0\|_\li,\qquad\nb\rho_0\in\hs(\R^2),\qquad \divv\,u_0=0.
    \]
    Let $f$ be a $C^{\ps+3}$-diffeomorphism on $[\rho_*,\rho^*]$. 
    
    Then, there exist a time $T=T\big(s,f',\rho_*,\rho^*,\|\nb\rho_0\|_\hs,\|u_0\|_\hs\big)>0$, depending only on the quantities inside the brackets, and a unique solution $(\rho,u,\nb\pi)$ to system \eqref{odd} on $[0,T]\times\R^2$, with initial data $(\rho,u)_{|t=0}=(\rho_0,u_0)$, such that 
    \begin{itemize}
        \item $\rho\in\li([0,T]\times\R^2)$ with $\rho(t,x)\in[\rho_*,\rho^*]$ for all $(t,x)\in[0,T]\times\R^2$, $\rho-\rz\in C^1([0,T],\hs(\R^2))$, $\nb\rho\in C([0,T],\hs(\R^2))\cap C^1([0,T],\hsm(\R^2))$;
        \item $u\in C([0,T],\hs(\R^2))\cap C^1([0,T],\hsm(\R^2))$;
        \item $\nb\pi\in C([0,T],\hsmd(\R^2))$, $\nb\Pi\in C([0,T],\hs(\R^2))$, where $\Pi:=\pi-f(\rho)\om$, and $\om:=\curl(u)=\pt_1u_2-\pt_2u_1$.
    \end{itemize}
    Moreover, after defining 
    \begin{equation}
        \label{g}
        g(\rho) := \int_{\rho_*}^\rho \frac{2}{r}f'(r)\dd r,\qquad\forall\,\rho\in[\rho_*,\rho^*],
    \end{equation}
    and $U:=u-\nbp g(\rho)$, the quadruple $(\rho,u,U,\nb\Pi)$ solves the Elsässer formulation \eqref{els} of the odd viscosity system.
    Additionally, the couples $(\rho,u)$ and $(\rho,U)$ satisfy the energy equalities
    \[
    \big\|\sqrt{\rho(t)}u(t)\big\|_{\ld(\R^2)} = \big\|\sqrt{\rz}u_0\big\|_{\ld(\R^2)}\qquad\text{and}\qquad\big\|\sqrt{\rho(t)}U(t)\big\|_{\ld(\R^2)} = \big\|\sqrt{\rz}U_0\big\|_{\ld(\R^2)},\qquad\forall\,t\in[0,T],
    \]
    where we have defined $U_0:=u_0-\nbp g(\rho_0)$.
\end{theorem}

\begin{remark}
    Theorem \ref{main} also applies in the particular case where $f$ is constant. Indeed, suppose that $f(\rho)=c$, for some $c\in\R$ and all $\rho\in[\rho_*,\rho^*]$. Under the assumptions of Theorem \ref{main}, one can reproduce the computations of Subsection \ref{sec:reformulation} to reformulate the odd viscosity system \eqref{odd} as
    \[
    \left\{
    \begin{aligned}
    & (\dt+u\cdot\nb)\rho = 0, \\
    & \rho(\dt+u\cdot\nb)u + \nabla\Pi = 0, \\
    & \divv\,u = 0,
    \end{aligned}
    \right.
    \]
    where $\Pi:=\pi-c\,\om$.
    
    This new system corresponds exactly to the classical non-homogeneous incompressible Euler equations, for which well-posedness results have been established, see \emph{e.g.} \cite{danchin2006,danchin2010,df,fanelli2012} and references therein.
    
    \medskip
    As a consequence, our result covers in particular all the viscosity coefficients of the form $f(\rho)=a\rho^\alpha+b$ for any $(a,b,\alpha)\in\R^3$, generalising the result of \cite{fgs} devoted to the case where $(a,b,\alpha)=(1,0,1)$. We also mention that the viscosity coefficients of the form $f(\rho)=\rho^\alpha$ are an important case of study in the literature for the compressible Navier--Stokes equations, see \emph{e.g.} \cite{vaigant1995,bvy2021}. 
\end{remark}

\subsection{Strategy of the proof}

We now present our strategy for proving Theorem \ref{main}. The idea is to follow the standard scheme of constructing a sequence of approximate solutions to the Elsässer formulation \eqref{els}, deriving uniform bounds for this sequence, and taking the limit. Unfortunately, this method does not apply directly here. 

To be more precise, let us consider the system
\begin{equation}
    \label{odd0}
    \left\{
\begin{aligned}
& (\dt+u\cdot\nb)\rho = 0, \\
& (\dt + U\cdot\nb)u + a\nabla\Pi = 0, \\
& \divv\,u = 0,
\end{aligned}
\right.
\end{equation}
where $a:=1/\rho$, and
\[
U := u - \nbp g(\rho),
\]
where $g$ is defined by \eqref{g}. Let us drop the time variable for a while, and focus only on the space regularity. Let $(\rz,\uz)$ be as in Theorem \ref{main}, and assume that we dispose of a triplet $(\rho^n,u^n,\nb\Pi^n)$, for some $n\ge0$, such that $(\nb\rho^n,u^n,\nb\Pi^n)\in\left(\hs(\R^2)\right)^3$. Now, define $\rho^{n+1}$ as the unique global-in-time solution to the transport equation 
\[
\left\{
\begin{aligned}
& (\dt+u^n\cdot\nb)\rho = 0, \\
& \rho_{|t=0} = \rz.
\end{aligned}
\right.
\]
We would like to construct a solution $(u^{n+1},\nb\Pi^{n+1})\in\left(\hs(\R^2)\right)^2$ to the system
\begin{equation}
    \label{momentn}
    \left\{
\begin{aligned}
&(\dt + U^n\cdot\nb)u + a^{n+1}\nabla\Pi = 0, \\
& \divv\,u = 0, \\
& u_{|t=0} = \uz,
\end{aligned}
\right.
\end{equation}
where $a^n:=1/\rho^n$, for some approximate effective velocity $U^n$ such that $(U^n)_{n\ge0}$ would converge to $U$ defined above in some suitable functional space, where $\rho$ and $u$ would be the limits of the sequences $(\rho^n)_{n\ge0}$ and $(u^n)_{n\ge0}$ being constructed. At this point, there are two possible choices for the definition of $U^n$. 

One can define it as 
\[
U^n := u^n - \nbp g(\rho^{n+1}).
\]
As the transport field $u^n$ only belongs to $\hs$, from Proposition \ref{trgrad}, one only gets that $\nb\rho^{n+1}\in\hsm$. Of course, this regularity is not sufficient to close the iterative argument. One could hope to improve it by means of the Elsässer formulation presented above, but this is not possible. Indeed, to get an equation for $U^n$, one needs to solve system \eqref{momentn} in the first place. But as we have only $U^n\in\hsm$, Theorem \ref{thstokes} only yields a solution $(u^{n+1},\nb\Pi^{n+1})\in\left(\hsm(\R^2)\right)^2$, which is again not enough.

The other possibility is to define 
\[
U^n := u^n - \nbp g(\rho^n).
\]
From the regularity properties on $(\rho^n,u^n)$, we have this time $U^n\in\hs$, and we can apply Theorem \ref{thstokes} to get a solution $(u^{n+1},\nb\Pi^{n+1})\in\left(\hs(\R^2)\right)^2$ to equation \eqref{momentn}. But with such a definition, even if we could find a transport equation for $U^n$ in the spirit of the fourth equation in \eqref{els} (which turns out not to be possible because of the shift of indices $n$ and $n+1$), this would provide no information on the regularity of $\nb\rho^{n+1}$, as it does not appear in this definition of $U^n$. Even with this choice, the iterative argument is thus again impossible to close.

\medskip
To tackle this difficulty, we proceed by viscous regularisation. More precisely, for $0<\eps\le1$, we consider the system 
\begin{equation}
    \label{oddeps0}
    \left\{
\begin{aligned}
& (\dt + u\cdot\nb)\rho = 0, \\
& (\dt + U\cdot\nb)u + a\nabla\Pi - \eps a\del u = 0, \\
& \divv\,u = 0.
\end{aligned}
\right.
\end{equation}
Such a regularisation provides a gain of two derivatives for the velocity $u$, which is more than enough to obtain $\nb\rho\in\hs$. We will then be able to prove the existence of a solution $(\re,\ue,\nb\Pe)$ to this new system, such that $(\nb\re,\ue,\nb\Pe)\in\left(\hs(\R^2)\right)^3$. Now, in order to find a solution $(\rho,u,\nb\Pi)$ to system \eqref{odd0} such that $(\nb\rho,u,\nb\Pi)\in\left(\hs(\R^2)\right)^3$, one has to bound the family $(\nb\re,\ue,\nb\Pe)_{0<\eps\le1}$ uniformly in $\left(\hs(\R^2)\right)^3$. Of course, as we only have $\eps\ue\in\hspd$, the estimate for $\ue$ in $\hsp$, hence the estimate for $\nb\re$ in $\hs$, is not uniform with respect to $0<\eps\le1$, and we are not able to take the limit as $\eps\to0$.

\medskip
This is where the Elsässer formulation \eqref{els} plays a crucial role. By defining 
\[
\ve := \ue - \nbp g(\re),
\]
we find that
\[
(\dt + \ue\cdot\nb)\ve = - \aeps\nb\Pe + \eps\aeps\del\ue .
\]
From the regularity properties of $(\re,\ue,\nb\Pe)$, we then gather from transport theory that $(\ve)_{0<\eps\le1}$ is uniformly bounded in $\hs$. Here comes into play the assumption made on $f$ in Theorem \ref{main}. Indeed, as $f$ is a $C^{\ps+3}$-diffeomorphism, then so is $g$, and Proposition \ref{compo2} allows us to transfer the $\hs$ regularity from $\ve$ to $\nb\re$, to finally gather that $(\nb\re)_{0<\eps\le1}$ is uniformly bounded in $\hs$. We then follow classical arguments to take the limit as $\eps\to0$ and obtain a solution $(\rho,u,\nb\Pi)$ to system \eqref{odd0}, with the claimed regularity properties, and satisfying the Elsässer formulation \eqref{els}. 

\medskip
Finally, we make use one last time of the Elsässer formulation to derive a stability result, that directly implies the uniqueness of solutions to \eqref{odd0}, hence of \eqref{odd}, in our functional framework.

\subsection*{Organisation of the paper} 

In Section \ref{sec:reg}, we construct a uniformly bounded family of solutions to system \eqref{oddeps0} on a fixed time interval. In Section \ref{sec:proofth}, we prove Theorem \ref{main}. We take the limit in the regularisation parameter to obtain solutions to the odd viscosity system \eqref{odd} and its Elsässer formulation \eqref{els}. We then prove the regularity properties and energy equalities stated in Theorem \ref{main}. To finish with, we derive a stability result, which in turn implies the uniqueness of the constructed solutions. In Appendix \ref{sec:stokes}, we provide a fundamental well-posedness result for a Navier--Stokes type system, and key estimates needed for the construction of the regularised solutions and the uniform bounds in Section \ref{sec:reg}. Finally, in Appendix \ref{sec:tools}, we recall some elements of Littlewood-Paley theory needed for our study.

\section{Existence and uniform bounds for a regularised system}
\label{sec:reg}

Let $0<\eps\le1$. We consider the system 
\begin{equation}
\label{oddeps}
\left\{
\begin{aligned}
& (\dt + u\cdot\nb)\rho = 0, \\
& \rho(\dt + U\cdot\nb)u + \nabla\Pi - \eps\del u = 0, \\
& \divv\,u = 0, \\
& (\rho,u)_{|t=0} = (\rez,\uzeps),
\end{aligned}
\right.
\end{equation}
where we have defined the divergence-free vector field 
\[
U := u - \nbp g(\rho),
\]
where $g$ is given by \eqref{g}. 

\medskip

This section is devoted to proving the following result. We refer to Appendix \ref{sec:timespace} for the definition of the time-space Besov spaces $\widetilde{L_T^q}H^\sig$.

\begin{theorem}
    \label{th:oddeps}
    Suppose that the assumptions of Theorem \ref{main} are satisfied. There exist a time $T_0=T_0\big(s,f',\rho_*,\rho^*,\|\nb\rho_0\|_\hs,\|u_0\|_\hs\big)>0$, depending only on the quantities inside the brackets, and a family $(\rho_\eps,u_\eps,\nb\Pi_\eps)_{0<\eps\le1}$ of solutions to \eqref{oddeps} on $[0,T_0]\times\R^2$, satisfying the uniform bounds 
    \begin{equation}
        \label{epsbound}
        \sup_{0<\eps\le1} \Big(\|\nb\re\|_{\widetilde{L_{T_0}^\infty}\hs} + \|\ue\|_{\widetilde{L_{T_0}^\infty}\hs} + \eps\|\ue\|_{\widetilde{L_{T_0}^1}\hspd} + \|\nb\Pe\|_{L_{T_0}^2\hsm}\Big) < \infty.
    \end{equation}
\end{theorem}

We prove Theorem \ref{th:oddeps} in three steps. First, for any $0<\eps\le1$, we construct a uniformly bounded sequence of approximate solutions $\big(\rho_\eps^n,u_\eps^n,\nb\Pi_\eps^n\big)_{n\ge0}$ to system \eqref{oddeps} on some time interval $[0,T_\eps]$. We then take the limit as $n\to+\infty$ to obtain a family of solutions $(\re,\ue,\nb\Pe)_{0<\eps\le1}$ to system \eqref{oddeps}. Finally, we construct a uniform time of existence $0<T_0\le T_\eps$, and prove the claimed uniform bounds for $(\re,\ue,\nb\Pe)_{0<\eps\le1}$ on $[0,T_0]$.

\subsection{Construction of a sequence of approximate solutions}
\label{subsec:approx}

Let $n\ge0$. We consider the system
\begin{equation}
    \label{eq:oddepsn}
    \left\{
    \begin{aligned}
    & (\dt + u^n\cdot\nb)\rho^{n+1} = 0, \\
    & \rho^{n+1}(\dt + U^n\cdot\nb)u^{n+1} + \nabla\Pi^{n+1} - \eps\del u^{n+1} = 0, \\
    & \divv\,u^{n+1} = 0, \\
    & (\rho^{n+1},u^{n+1})_{|t=0} = (\rez,\uzeps),
    \end{aligned}
    \right.
\end{equation}
where
\[
U^n := u^n - \nbp g(\rho^{n+1}).
\]
Let $T>0$ and $0<\rho_*\le\rho^*<\infty$. We introduce the space 
\[
\begin{split}
    E_T := \Big\{(\rho,u,\nb\Pi)\in&\li([0,T]\times\R^2)\times\big(\ctt\hs\cap\lutt H^{s+2}\big)\times \lutt\hs : \\
    &\rho(t,x)\in[\rho_*,\rho^*]\quad\forall\,(t,x)\in[0,T]\times\R^2,\quad\nb\rho\in\ctt\hs,\quad \divv\,u=0\Big\},
\end{split}
\]
endowed with the natural norm 
\[
\|(\rho,u,\nb\Pi)\|_{E_T} := \|\nb\rho\|_{\litt\hs} + \|u\|_{\litt\hs} + \|u\|_{\lutt H^{s+2}} + \|\nb\Pi\|_{\lutt\hs}.
\]
We prove the following statement.
\begin{proposition}
    \label{p:oddepsnbounds}
    Suppose that the assumptions of Theorem \ref{main} are satisfied. For any $0<\eps\le1$, there exist a time $T_\eps=T_\eps\big(s,f',\rho_*,\rho^*,\|\nb\rho_0\|_\hs,\|u_0\|_\hs\big)>0$, depending only on $\eps$ and the quantities inside the brackets, and a sequence $\big(\rho_\eps^n,u_\eps^n,\nb\Pi_\eps^n\big)_{n\ge0}\in E_{T_\eps}$ of solutions to system \eqref{eq:oddepsn} on $[0,T_\eps]\times\R^2$, satisfying the uniform estimate
    \begin{equation}
        \label{nbound}
        \sup_{n\ge0}\;\|(\rho_\eps^n,u_\eps^n,\nb\Pi_\eps^n)\|_{E_{T_\eps}} < \infty,\qquad \forall\,0<\eps\le1.
    \end{equation}
\end{proposition}

As $0<\eps\le1$ is fixed, we simply denote $T\equiv T_\eps$ and $(\rho^n,u^n,\nb\Pi^n)\equiv(\rho_\eps^n,u_\eps^n,\nb\Pi_\eps^n)$, for any $n\ge0$.

\begin{proof}
For any $t\ge0$, let us consider the bounds
    \begin{equation}
        \label{t1}
        C\Big(1 + Ce^{4C^2\|\uzeps\|_\hs}\|\nb\rez\|_\hs\Big)^\lambda t\Big(4C\|\uzeps\|_\hs+Ce^{4C^2\|\uzeps\|_\hs}\|\nb\rez\|_\hs\Big) \le \log2,
    \end{equation}
    \begin{equation}
        \label{t2}
        2C\eps\Big(1 + Ce^{4C^2\|\uzeps\|_\hs}\|\nb\rez\|_\hs\Big)^\lambda t \le \frac{1}{2},
    \end{equation}
    \begin{equation}
        \label{t3}
        \begin{split}
        C\Big(1 + Ce^{4C^2\|\uzeps\|_\hs}\|\nb\rez\|_\hs\Big)^\lambda \Big(\big[t(4C\|\uzeps\|_\hs &+ Ce^{4C^2\|\uzeps\|_\hs}\|\nb\rez\|_\hs) + \eps(t^{1/2}+t)\big]4C\|\uzeps\|_\hs \\
        &+ t^{1/6}4C\|\uzeps\|_\hs\Big) \le 1,
        \end{split}
    \end{equation}
    for some constants $C=C\big(\eps,s,f',\rho_*,\rho^*\big)\ge1$ and $\lambda=\lambda(s)>0$ sufficiently large, depending only on the quantities inside the brackets, to be precised later. Let us now define the time 
    \[
    T := \sup\Big\{t\ge0 : \eqref{t1}-\eqref{t3}\;{\rm are\;satisfied}\Big\}.
    \]
    For any $n\ge0$, let us consider the bounds
    \begin{equation}
        \label{rhon}
        \rho^n(t,x)\in[\rho_*,\rho^*]\quad\forall\,(t,x)\in[0,T]\times\R^2,\qquad \|\nb\rho^n\|_{\litt\hs} \le Ce^{4C^2\|\uzeps\|_\hs}\|\nb\rez\|_\hs,
    \end{equation}
    \begin{equation}
        \label{un}
        \U^n(T) := \|u^n\|_{\litt\hs} + \eps\|u^n\|_{\lutt\hspd} \le 4C\|\uzeps\|_\hs,
    \end{equation}
    \begin{equation}
        \label{nbpin}
        \|\nb\Pi^n\|_{\lutt\hs} \le 1.
    \end{equation}
    Consider the triplet
    \[
    (\rho^0,u^0,\nb\Pi^0) := (\bar\rho_0,0,0),
    \]
    where we have defined 
    \[
    \bar\rho_0(t,x):=\rho_0(x),\qquad\forall\,(t,x)\in[0,T]\times\R^2.
    \]
    Obviously, $(\rho^0,u^0,\nb\Pi^0)$ belongs to $E_T$, and satisfies the bounds \eqref{rhon}-\eqref{nbpin} for $n=0$.
    
    \medskip
    Let $n\ge0$. Assume that there exists a triplet $(\rho^n,u^n,\nb\Pi^n)\in E_T$, satisfying the bounds \eqref{rhon}-\eqref{nbpin}. Denote by $\psi^n$ the flow of $u^n$, defined, for all $(t,x)\in[0,T]\times\R^2$, by
    \[
    \psi^n(t,x) \equiv \psi_t^n(x) := x + \int_0^t u^n\big(\tau,\psi_\tau^n(x)\big)\dd\tau.
    \]
    Since $u^n\in\ct\hs$ and $\hsm\into C_b$, we have $\psi^n\in C^1([0,T]\times\R^2)$. As $\psi_t^n$ is a diffeomorphism over $\R^2$ for any $t\in[0,T]$, we can define
    \[
    \rho^{n+1}(t,x) := \rho_0\big((\psi_t^n)^{-1}(x)\big),\qquad \forall\,(t,x)\in[0,T]\times\R^2.
    \]
    From this expression, it is clear that $\rho^{n+1}$ belongs to $C^1([0,T]\times\R^2)$, and satisfies the bounds
    \begin{equation}
        \label{rhonpbound}
        \rho^{n+1}(t,x)\in[\rho_*,\rho^*],\qquad\forall\,(t,x)\in[0,T]\times\R^2,
    \end{equation}
    and the equation 
    \[
    (\dt + u^n\cdot\nb)\rho^{n+1} = 0\qquad\mathrm{on}\;[0,T]\times\R^2.
    \]
    From Proposition \ref{trgrad}, \eqref{timespace2} and \eqref{un}, we have
    \begin{equation}
        \label{nbrhonp}
        \|\nb\rho^{n+1}\|_{\litt\hs} \le Ce^{C\|u^n\|_{\lutt\hspd}}\|\nb\rez\|_\hs \le Ce^{4C^2\|\uzeps\|_\hs}\|\nb\rez\|_\hs,
    \end{equation}
    for some constant $C=C(\eps,s,\rho_*,\rho^*)\ge1$, depending only on the quantities inside the brackets. This proves that $\rho^{n+1}$ satisfies \eqref{rhon}. Note that in the following computations, as the constant $C$ will keep changing, we will only indicate when a new parameter is involved in its dependency, and we will not rename it. 
    
    \medskip
    Now, after defining 
    \begin{equation}
        \label{anp}
        a^{n+1} := \frac{1}{\rho^{n+1}},\qquad a_* := \frac{1}{\rho^*},\qquad a^* := \frac{1}{\rho_*}, 
    \end{equation}
    we have $a^{n+1}(t,x)\in[a_*,a^*]$, for all $(t,x)\in[0,T]\times\R^2$. Moreover, in view of Proposition \ref{compo2}, we have $\nb a^{n+1}\in\litt\hs$, together with the estimate 
    \begin{equation}
        \label{nbanbrho}
        \|\nb a^{n+1}\|_{\litt\hs} \lesssim \|\nb\rho^{n+1}\|_{\litt\hs}. 
    \end{equation}
    We then define the divergence-free vector field
    \[
    U^n := u^n - \nbp g(\rho^{n+1}).
    \]
    From \eqref{un}, \eqref{nbrhonp} and Proposition \ref{compo2}, we have $U^n\in\litt\hs$. From \eqref{un}, \eqref{timespace2}, \eqref{timespace1} and \eqref{nbrhonp}, it holds that $U^n\in\lut\hs$. In view of all these properties, we are able to apply Theorem \ref{thstokes} to obtain a solution $(u^{n+1},\nb\Pi^{n+1})\in(\ctt\hs\cap\lutt\hspd)\times\lutt\hs$ on $[0,T]\times\R^2$ to the system
    \begin{equation}
    \label{stokes1n}
    \left\{
    \begin{aligned}
    & (\dt + U^n\cdot\nb)u + a^{n+1}\nb\Pi - \eps a^{n+1}\del u = 0, \\
    & \divv\,u = 0, \\
    & u_{|t=0}=\uzeps,
    \end{aligned}
    \right.
    \end{equation}
    satisfying the bounds
    \begin{equation}
        \label{unpt}
        \|u^{n+1}\|_{\litt\hs} + \eps\|u^{n+1}\|_{\lutt\hspd} \le Ce^{C\B_{n,T}^\lambda\|U^n\|_{\lut\hs}} \Big(\|\uzeps\|_\hs + \eps\B_{n,T}^\lambda T\|u^{n+1}\|_{\litt\hs}\Big),
    \end{equation}
    \begin{equation}
        \label{nbpinpt}
        \|\nb\Pi^{n+1}\|_{\lutt\hs} \le C\B_{n,T}^\lambda \Big(\big(T\|U^n\|_{\litt\hs} + \eps(T^{1/2}+T)\big)\|u^{n+1}\|_{\litt\hs} + \eps T^{1/6}\|u^{n+1}\|_{\lutt\hspd}\Big).
    \end{equation}
    where we have defined
    \[
    \B_{n,T} := 1 + \|\nb\rho^{n+1}\|_{\litt\hs},
    \]
    and made use of inequality \eqref{nbanbrho}. In view of \eqref{nbrhonp}, we have 
    \begin{equation}
        \label{bnt}
        \B_{n,T} \le 1 + Ce^{4C^2\|\uzeps\|_\hs}\|\nb\rez\|_\hs.
    \end{equation}
    Using now \eqref{timespace1}, Proposition \ref{compo2}, \eqref{un} and \eqref{nbrhonp}, we have 
    \begin{equation}
        \label{Un}
        \|U^n\|_{\lut\hs} \le T\|U^n\|_{\litt\hs} \lesssim T\big(4C\|\uzeps\|_\hs + Ce^{4C^2\|\uzeps\|_\hs}\|\nb\rez\|_\hs\big),
    \end{equation}
    for a new constant $C=C\big(\eps,s,f',\rho_*,\rho^*\big)\ge1$. Plugging these estimates in \eqref{unpt}, we gather that
    \begin{align}
        \label{eq:energ1}
        \U^{n+1}(T) \le \;&Ce^{C(1 + Ce^{4C^2\|\uzeps\|_\hs}\|\nb\rez\|_\hs)^\lambda T(4C\|\uzeps\|_\hs+Ce^{4C^2\|\uzeps\|_\hs}\|\nb\rez\|_\hs)} \\
        \nonumber&\times\Big(\|\uzeps\|_\hs + \eps\big(1 + Ce^{4C^2\|\uzeps\|_\hs}\|\nb\rez\|_\hs\big)^\lambda T\U^{n+1}(T)\Big).
    \end{align}
    From the bounds \eqref{t1} and \eqref{t2}, we finally have 
    \begin{equation}
        \label{unp}
        \U^{n+1}(T) \le 4C\|\uzeps\|_\hs.
    \end{equation}
    This proves that $\U^{n+1}$ satisfies \eqref{un}. Now, in view of \eqref{nbpinpt}, \eqref{Un} and \eqref{unp}, we deduce from \eqref{t3} that 
    \[
    \|\nb\Pi^{n+1}\|_{\lutt\hs} \le 1,
    \]
    and $\nb\Pi^{n+1}$ satisfies \eqref{nbpin}. This completes the iterative argument. 
\end{proof}

\subsection{Convergence of the sequence towards a regularised solution}
\label{convapprox}

With Proposition \ref{p:oddepsnbounds} at hand, we can now prove the existence of solutions to the regularised system \eqref{oddeps}.

\begin{proposition}
    \label{p:oddepsnconv}
    Suppose that the assumptions of Theorem \ref{main} are satisfied. For any $0<\eps\le1$, there exists a solution $(\re,\ue,\nb\Pe)\in E_{T_\eps}$ to system \eqref{oddeps} on $[0,T_\eps]\times\R^2$, where $T_\eps$ is given by Proposition \ref{p:oddepsnbounds}.
\end{proposition}

As before, we simply denote $T\equiv T_\eps$ and $(\rho^n,u^n,\nb\Pi^n)\equiv(\rho_\eps^n,u_\eps^n,\nb\Pi_\eps^n)$, for any $n\ge0$.

\begin{proof} We wish to take the limit as $n\to+\infty$ in system \eqref{eq:oddepsn} in some suitable functional space. 

\medskip

We prove stability in $C_TL^2$. For $n,m\ge0$, let us denote 
    \[
    (\delta h)^{n,m}:=h^{n+m}-h^n,\qquad\text{for}\;h\in\{\rho,u,U,\nb\Pi\},
    \]
    and define the energy function 
    \[
    \e^{n,m}(t) := \|(\delta u)^{n,m}(t)\|_\ld^2 + \|(\delta\rho)^{n,m}(t)\|_\ld^2.
    \]
    
    On the one hand, for all $k\ge1$, 
    \begin{equation}
        \label{eqrhok}
        \pt_t\rho^k + u^{k-1}\cdot\nb\rho^k = 0.
    \end{equation}
    Using this with $k=n+m$, $k=n$ and taking the difference, we infer
    \[
    \pt_t(\delta\rho)^{n,m} + u^{n+m-1}\cdot\nb(\delta\rho)^{n,m} = -(\delta u)^{n-1,m}\cdot\nb\rho^n.
    \]
    Multiplying this equation by $(\delta\rho)^{n,m}$, integrating in time and using \eqref{nbound}, we gather 
    \begin{equation}
        \label{rhobound}
        \frac{1}{2}\ddt \|(\delta\rho)^{n,m}\|_\ld^2 \lesssim \|(\delta u)^{n-1,m}\|_\ld \|\nb\rho^n\|_\li \|(\delta\rho)^{n,m}\|_\ld \lesssim \e^{n-1,m} + \|\nb\rho^n\|_\li^2\e^{n,m},
    \end{equation}
    where we have also used \eqref{nbound}.  

    On the other hand, for all $k\ge1$, 
    \begin{equation}
        \label{equk}
        \pt_tu^k + (U^{k-1}\cdot\nb)u^k + \frac{1}{\rho^k}\nb\Pi^k - \eps\frac{1}{\rho^k}\del u^k = 0.
    \end{equation}
    Decomposing 
    \[
    (U^{k-1}\cdot\nb)u^k = (u^{k-1}\cdot\nb)u^k - (\nbp g(\rho^k)\cdot\nb)u^k
    \]
    and using that 
    \[
    (\nbp g(\rho^k)\cdot\nb)u^k = - (\nb g(\rho^k)\cdot\nbp)u^k = - \divv(g(\rho^k)\nbp u^k),
    \]
    we have 
    \[
    \pt_tu^k + (u^{k-1}\cdot\nb)u^k + \frac{1}{\rho^k}\nb\Pi^k - \eps\frac{1}{\rho^k}\Delta u^k + \divv(g(\rho^k)\nbp u^k) = 0,
    \]
    for all $k\ge1$. Using this with $k=n+m$, $k=n$, taking the difference and multiplying by $\rho^{n+m}$, we infer 
    \begin{equation}
        \label{deltaunm}
        \begin{split}
            &\rho^{n+m}\pt_t(\delta u)^{n,m} + \rho^{n+m}u^{n+m-1}\cdot\nb(\delta u)^{n,m} + (\delta\nb\Pi)^{n,m} - \eps\del(\delta u)^{n,m} + \rho^{n+m}\divv(g(\rho^{n+m})\nbp(\delta u)^{n,m}) \\
            &= -\rho^{n+m}((\delta u)^{n-1,m}\cdot\nb)u^n + \frac{(\delta\rho)^{n,m}}{\rho^n}\nb\Pi^n - \eps\frac{(\delta\rho)^{n,m}}{\rho^n}\del u^n - \rho^{n+m}\divv(\delta g(\rho)^{n,m}\nbp u^n).   
        \end{split}
    \end{equation}
    Using that 
    \[
    \begin{split}
        \int\rho^{n+m}\pt_t(\delta u)^{n,m}\cdot(\delta u)^{n,m} &= \frac{1}{2}\frac{\dd}{\dd t}\|\sqrt{\rho^{n+m}}(\delta u)^{n,m}\|_\ld^2 - \frac{1}{2}\int\pt_t\rho^{n+m}|(\delta u)^{n,m}|^2, \\
        \int\rho^{n+m}(u^{n+m-1}\cdot\nb)(\delta u)^{n,m}\cdot(\delta u)^{n,m} &= -\frac{1}{2}\int(u^{n+m-1}\cdot\nb\rho^{n+m})|(\delta u)^{n,m}|^2,
    \end{split}
    \]
    we infer 
    \[
    \begin{split}
        &\int\left(\rho^{n+m}\pt_t(\delta u)^{n,m} + \rho^{n+m}(u^{n+m-1}\cdot\nb)(\delta u)^{n,m}\right)\cdot(\delta u)^{n,m} \\ 
        &= \frac{1}{2}\frac{\dd}{\dd t}\|\sqrt{\rho^{n+m}}(\delta u)^{n,m}\|_\ld^2 - \frac{1}{2}\int\left(\pt_t\rho^{n+m}+u^{n+m-1}\cdot\nb\rho^{n+m}\right)|(\delta u)^{n,m}|^2 \\
        &= \frac{1}{2}\frac{\dd}{\dd t}\|\sqrt{\rho^{n+m}}(\delta u)^{n,m}\|_\ld^2,
    \end{split}
    \]
    where we have also used equation \eqref{eqrhok} with $k=n+m$ and the divergence-free condition on $(\delta u)^{n,m}$. The pressure and odd viscosity terms on the left-hand side of equation \eqref{deltaunm} being orthogonal to $(\delta u)^{n,m}$, we obtain
    \[
    \begin{split}
        \frac{1}{2}\ddt \|\sqrt{\rho^{n+m}}(\delta u)^{n,m}\|_\ld^2 + \eps\|\nb(\delta u)^{n,m}\|_\ld^2 \lesssim &\|(\delta u)^{n-1,m}\|_\ld\|\nb u^n\|_\li\|(\delta u)^{n,m}\|_\ld \\
        &+ \|(\delta\rho)^{n,m}\|_\ld\|\nb\Pi^n\|_\li\|(\delta u)^{n,m}\|_\ld \\ 
        &+ \|(\delta\rho)^{n,m}\|_\ld\|\del u^n\|_\li\|(\delta u)^{n,m}\|_\ld \\
        &+ \|\delta g(\rho)^{n,m}\|_\ld\|\nb u^n\|_\li\|\nb(\delta u)^{n,m}\|_\ld,
    \end{split}
    \]
    where we have performed an integration by parts for the last term. Now, we focus on the last term on the right-hand side. As $g'\in\li$, we get from the mean value theorem that
    \[
    \|\delta g(\rho)^{n,m}\|_\ld \lesssim \|(\delta\rho)^{n,m}\|_\ld.
    \]
    Using now the Young inequality, we deduce that 
    \begin{equation}
        \label{ubound}
        \frac{1}{2}\ddt\|\sqrt{\rho^{n+m}}(\delta u)^{n,m}\|_\ld^2 \lesssim \e^{n-1,m} + \Big(\|\nb u^n\|_\li^2 + \|\nb\Pi^n\|_\li + \|\del u^n\|_\li\Big)\e^{n,m}.
    \end{equation}
    Summing up estimates \eqref{rhobound}, \eqref{ubound} and using the Young inequality, one gathers that 
    \[
    \ddt\left(\|(\delta\rho)^{n,m}\|_\ld^2 + \|\sqrt{\rho^{n+m}}(\delta u)^{n,m}\|_\ld^2\right) \lesssim \e^{n-1,m} + \Big(\|\nb\rho^n\|_\li^2 + \|\nb u^n\|_\li^2 + \|\del u^n\|_\li + \|\nb\Pi^n\|_\li\Big)\e^{n,m}.
    \]
    Integrating in time, we obtain 
    \[
    \e^{n,m}(t) \le C\int_0^t \e^{n-1,m}\dd\tau + C\int_0^t \Big(\|\nb\rho^n\|_\li^2 + \|\nb u^n\|_\li^2 + \|\del u^n\|_\li + \|\nb\Pi^n\|_\li\Big)\e^{n,m}\dd\tau,
    \]
    for some constant $C>0$. Using the Grönwall inequality, we infer
    \[
    \e^{n,m}(t) \le \Big(C\int_0^t \e^{n-1,m}\dd\tau\Big)\exp\left(C\int_0^t \Big(\|\nb\rho^n\|_\li^2 + \|\nb u^n\|_\li^2 + \|\del u^n\|_\li + \|\nb\Pi^n\|_\li\Big)\dd\tau\right).
    \]
    Now, as $s>2$, we have $\hsm\into\li$, and 
    \[
    \|\nb\rho^n\|_{L_T^2\li} \lesssim \|\nb\rho^n\|_{L_T^2H^s} \lesssim \sqrt{T}\|\nb\rho^n\|_{\lit H^s} \lesssim \sqrt{T}\|\nb\rho^n\|_{\litt H^s},
    \]
    \[
    \|\nb\Pi^n\|_{\lut\li} \lesssim \|\nb\Pi^n\|_{\lut H^{s-1}} \lesssim \|\nb\Pi^n\|_{\lutt H^s},
    \]
    \[
    \|\del u^n\|_{\lut\li} \lesssim \|\del u^n\|_{\lut H^{s-1}} \lesssim \|u^n\|_{\lut H^{s+1}} \lesssim \|u^n\|_{\lutt H^{s+2}},
    \] 
    \[
    \|\nb u^n\|_{L_T^2\li} \lesssim \|\nb u^n\|_{L_T^2H^s} = \|\nb u^n\|_{\widetilde{L_T^2}H^s} \le \|\nb u^n\|_{\litt H^{s-1}}^{1/2}\|\nb u^n\|_{\lutt H^{s+1}}^{1/2} \lesssim \|u^n\|_{\litt H^s} + \|u^n\|_{\lutt H^{s+2}},
    \]
    where we have used \eqref{timespace1}-\eqref{timespace2}, and the interpolation inequality \eqref{interpolation} for the last estimate. From \eqref{nbound}, we then obtain that
    \[
    \e^{n,m}(t) \le C_T\int_0^t\e^{n-1,m}(\tau)\dd\tau,
    \]
    Defining 
    \[
    \f^n(t):=\sup_{m\ge0}\sup_{0\le\tau\le t}\e^{n,m}(\tau),
    \]
    one has, for all $0\le t\le T$,
    \[
    \f^n(t)\le C_T\int_0^t \f^{n-1}(\tau)\dd\tau,
    \]
    thus
    \[
    \f^n(t)\le \frac{(C_TT)^n}{n!}\f^0(t).
    \]
    This shows that $(\rho^n-\rho_0)_{n\ge0}$ and $(u^n)_{n\ge0}$ are Cauchy sequences in $C_T\ld$. Furthermore, as they are bounded in $C_TH^s$, one gathers by interpolation that they are Cauchy sequences in $C_TH^\sig$ for all $\sig<s$: there exist $r_\eps\in C_TH^\sig$ and $\ue\in C_TH^\sig$ such that we have the strong convergences 
    \begin{equation}
        \label{conv:run}
        \rho^n-\rho_0 \to r_\eps\qquad\text{and}\qquad u^n\to\ue\qquad\text{in}\;C_TH^\sig,\qquad\text{as}\;n\to+\infty.
    \end{equation}
    Setting now $\re:=r_\eps+\rho_0$, it follows that 
    \begin{equation}
        \label{conv:rhon}
        \rho^n-\rho_0 \to \re-\rho_0\qquad\text{in}\;C_TH^\sig,\qquad\text{as}\;n\to+\infty,
    \end{equation}
    and 
    \[
    \nb\rho^n \to \nb\re\qquad\text{in}\;C_TH^{\sig-1},\qquad\text{as}\;n\to+\infty.
    \]
    Then, from the mean value theorem, one has 
    \[
    \begin{split}
        \|\nb g(\rho^n)-\nb g(\re)\|_\ld &\le \|g'(\rho^n)\|_\li \|\nb\rho^n-\nb\re\|_\ld + \|g'(\rho^n)-g'(\re)\|_\ld \|\nb\re\|_\li \\
        &\le \|g'(\rho^n)\|_\li \|\nb\rho^n-\nb\re\|_\ld + \|g''\|_{\li([\rho_*,\rho^*])}\|\rho^n-\re\|_\ld \|\nb\re\|_\li,
    \end{split}
    \]
    which proves the strong convergence $\nb g(\rho^n)\to\nb g(\re)$ in $C_T\ld$, as $n\to+\infty$. On the other hand, we deduce from \eqref{nbound} and Proposition \ref{compo2} that $(\nb g(\rho^n))_{n\ge0}$ is uniformly bounded in $C_T\hs$. We then gather by interpolation that for any $\sig<s$,
    \[
    \nb g(\rho^n) \to \nb g(\re)\qquad\text{in}\;C_TH^\sig,\qquad\text{as}\;n\to+\infty.
    \]
    It follows that 
    \begin{equation}
        \label{conv:Un}
        U^n \to \ve\qquad\text{in}\;C_TH^\sig,\qquad\text{as}\;n\to+\infty,
    \end{equation}
    where we have defined 
    \begin{equation}
        \label{eq:Ue}
        \ve:=\ue-\nbp g(\re).
    \end{equation}
    Coming back to \eqref{conv:run}, we have, for all $\sig<s$,
    \[
    \divv\,u^n\to\divv\,\ue\qquad\mathrm{in}\;\ct H^{\sig-1},\qquad\mathrm{as}\;n\to+\infty.
    \]
    From this and the third equation in \eqref{eq:oddepsn}, we gather that $\divv\,\ue=0$ almost everywhere on $[0,T]\times B_R$. As $\ue$ is continuous on $[0,T]\times\R^2$, we deduce that $\divv\,\ue=0$ everywhere on $[0,T]\times\R^2$.
    
    Let us now study the convergence of $(\nb\Pi^n)_{n\ge0}$. Coming back to equation \eqref{equk} with $k=n+m$, $k=n$ and taking the difference, we have 
    \[
    \begin{split}
        &\pt_t(\delta u)^{n,m} + (U^{n+m-1}\cdot\nb)(\delta u)^{n,m} + \frac{1}{\rho^{n+m}}(\delta\nb\Pi)^{n,m} - \eps\frac{1}{\rho^{n+m}}\del(\delta u)^{n,m}  \\
        &= -((\delta U)^{n-1,m}\cdot\nb)u^n + \frac{(\delta\rho)^{n,m}}{\rho^{n+m}\rho^n}\nb\Pi^n - \eps\frac{(\delta\rho)^{n,m}}{\rho^{n+m}\rho^n}\del u^n.   
    \end{split}
    \]
    Then, one has the elliptic equation 
    \[
    \begin{split}
        -\divv\left(\frac{1}{\rho^{n+m}}(\delta\nb\Pi)^{n,m}\right) = \divv\Big(&(U^{n+m-1}\cdot\nb)(\delta u)^{n,m} - \eps\frac{1}{\rho^{n+m}}\del(\delta u)^{n,m} + ((\delta U)^{n-1,m}\cdot\nb)u^n \\
        &- \frac{(\delta\rho)^{n,m}}{\rho^{n+m}\rho^n}\nb\Pi^n + \eps\frac{(\delta\rho)^{n,m}}{\rho^{n+m}\rho^n}\del u^n\Big).
    \end{split}
    \]
    From Lemma \ref{ellipticlow}, we obtain 
    \[
    \begin{split}
        \|(\delta\nb\Pi)^{n,m}\|_\ld \lesssim &\|U^{n+m-1}\|_\li \|(\delta u)^{n,m}\|_{H^1} + \|(\delta u)^{n,m}\|_{H^2} + \|(\delta U)^{n-1,m}\|_\ld \|\nb u^n\|_\li \\
        &+ \|(\delta\rho)^{n,m}\|_\ld \|\nb\Pi^n\|_\li + \|(\delta\rho)^{n,m}\|_\ld \|\del u^n\|_\li.
    \end{split}
    \]
    Integrating in time and using that $\hsm\into\li$ as well as \eqref{timespace2},
    \[
    \begin{split}
        \|(\delta\nb\Pi)^{n,m}\|_{\lut\ld} \lesssim &T\|U^{n+m-1}\|_{\litt\hs} \|(\delta u)^{n,m}\|_{\lit H^1} + T\|(\delta u)^{n,m}\|_{\lit H^2} + \|(\delta U)^{n-1,m}\|_{\lit\ld} \|u^n\|_{\lut\hspd} \\
        &+ \|(\delta\rho)^{n,m}\|_{\lit\ld} \|\nb\Pi^n\|_{\lutt\hs} + \|(\delta\rho)^{n,m}\|_{\lit\ld} \|u^n\|_{\lutt\hspd}.
    \end{split}
    \]
    From the previous boundedness and convergence properties, we deduce that $(\nb\Pi^n)_{n\ge0}$ is a Cauchy sequence in $\lut\ld$. Since it is also bounded in $\lutt\hs$, we deduce by \eqref{timespace2} that $(\nb\Pi^n)_{n\ge0}$ is a Cauchy sequence in $\lut\hsig$, for all $\sig<s$: there exists $\nb\Pe\in\lut\hsig$ such that 
    \begin{equation}
        \label{conv:nbPin}
        \nb\Pi^n \to \nb\Pe \qquad \text{in} \; \lut\hsig, \qquad \text{as} \; n\to+\infty.
    \end{equation}
    
    In view of the convergence properties \eqref{conv:run}, \eqref{conv:rhon}, \eqref{conv:Un} and \eqref{conv:nbPin}, we can now take the limit as $n\to+\infty$ in system \eqref{eq:oddepsn} to deduce that the triplet $(\rho_\eps,u_\eps,\nb\Pi_\eps)$ solves system \eqref{oddeps} on $[0,T]\times\R^2$:
    \begin{equation}
\label{oddepsbis}
\left\{
\begin{aligned}
& (\dt + \ue\cdot\nb)\re = 0, \\
& \re(\dt + \ve\cdot\nb)\ue + \nabla\Pe - \eps\del\ue = 0, \\
& \divv\,\ue = 0, \\
& (\re,\ue)_{|t=0} = (\rez,\uzeps).
\end{aligned}
\right.
\end{equation}
From the first equation, we also deduce that $\rho_*\le\re\le\rho^*$ on $[0,T]\times\R^2$. Moreover, this solution belongs to $E_T$, as can be seen using \eqref{nbound} and the Fatou property of time-space Besov spaces (see Theorem \ref{th:fatou}).
\end{proof}

\subsection{Uniform bounds for the regularised solutions}
\label{subsec:unifbounds}

We conclude this section by proving the last part of Theorem \ref{th:oddeps}.

\medskip
Let $0<\eps\le1$. Consider the time $T_\eps$ and the couple $(\re,\ue)$ given by Proposition \ref{p:oddepsnconv}. For any $t\in(0,T_\eps]$, we denote 
\[
\eeps(t) := \|\nb\re\|_{\widetilde{L_t^\infty}\hs} + \|\ue\|_{\widetilde{L_t^\infty}\hs} + \eps\|\ue\|_{\widetilde{L_t^1}\hspd}.
\]
We prove the following statement.

\begin{proposition}
\label{epsbounds}
    Suppose that the assumptions of Theorem \ref{main} are satisfied. Then, there exists a time $T_0=T_0\big(s,f',\rho_*,\rho^*,\|\nb\rho_0\|_\hs,\|u_0\|_\hs\big)>0$, depending only on the quantities inside the brackets, such that 
    \[
    \inf_{0<\eps\le1} T_\eps\ge T_0>0\qquad\text{and}\qquad\sup_{0<\eps\le1} \Big(\eeps(T_0) + \|\nb\Pe\|_{L_{T_0}^2\hsm}\Big) < \infty.
    \]
\end{proposition}

\begin{proof} 
Since $f$ is a $C^{\ps+3}$-diffeomorphism on $[\rho_*,\rho^*]$, then so is $g$, and in view of the relation \eqref{eq:Ue} and Proposition \ref{compo2}, we have for any $T\in(0,T_\eps]$,
\begin{equation}
    \label{eeps}
    \eeps(T) \lesssim \|\ue\|_{\litt\hs} + \eps\|\ue\|_{\lutt\hspd} + \|\ve\|_{\litt\hs}.
\end{equation}

\medskip
We now estimate the quantities appearing on the right-hand side. 

\medskip
From system \eqref{oddepsbis} and the above properties, we can reproduce the computations of Subsection \ref{sec:reformulation} to gather that 
\begin{equation}
    \label{elseps}
    \left\{
    \begin{aligned}
        & (\dt+\ue\cdot\nb)\re = 0, \\
        & \re(\dt + \ve\cdot\nb)\ue + \nb\Pe - \eps\del\ue = 0, \\
        & \re(\dt + \ue\cdot\nb)\ve + \nb\Pe - \eps\del\ue = 0, \\
        & \divv\,\ue = \divv\,\ve = 0.
    \end{aligned}
\right.
\end{equation}
The second equation reads
\[
(\dt + \ve\cdot\nb)\ue + \aeps\nb\Pe - \eps\aeps\del\ue = 0,
\]
where $\aeps:=1/\re$. Define
\[
\B_{\eps,T} := 1 + \|\nb\re\|_{\litt\hs}.
\]
As $(\re,\ue)\in E_T$, we deduce from Theorem \ref{thstokes} that $\ue$ satisfies the bound
\[
\|\ue\|_{\litt\hs} + \eps\|\ue\|_{\lutt\hspd} \le Ce^{CT\B_{\eps,T}^\lambda\|\ve\|_{\litt\hs}} \Big(\|\uzeps\|_\hs + \eps\B_{\eps,T}^\lambda T\|\ue\|_{\litt\hs}\Big),
\]
where we have also made use of Proposition \ref{compo2}. Setting
\[
E_0 := \|\nb\rez\|_\hs + \|\uzeps\|_\hs,
\]
we then have
\begin{equation}
    \label{ue}
    \|\ue\|_{\litt\hs} + \eps\|\ue\|_{\lutt\hspd} \le Ce^{CT(1+\eeps(T))^\lambda\eeps(T)} \left(E_0 + \eps(1+\eeps(T))^\lambda T\eeps(T)\right).
\end{equation}
Furthermore, we can bound $\nb\Pe$ as 
\[
\|\nb\Pe\|_{\lutt\hs} \le C\B_{\eps,T}^\lambda \Big(\big(T\|\ve\|_{\litt\hs} + \eps(T^{1/2}+T)\big)\|\ue\|_{\litt\hs} + \eps T^{1/6}\|\ue\|_{\lutt\hspd}\Big).
\]
It follows that 
\begin{equation}
    \label{nbPe}
    \|\nb\Pe\|_{\lutt\hs} \le C(1+\eeps(T))^\lambda \Big(\big(T\eeps(T) + \eps(T^{1/2}+T)\big)\eeps(T) + T^{1/6}\eeps(T)\Big).
\end{equation}

The third equation in \eqref{elseps} reads 
\[
(\dt + \ue\cdot\nb)\ve = -\aeps\nb\Pe + \eps\aeps\del\ue.
\]
Define 
\[
\vzeps := \uzeps - \nbp g(\rez).
\]
From the regularity properties \eqref{epsbound}, we can apply Theorem \ref{th:transport1} to gather that 
\[
\|\ve\|_{\litt\hs} \le Ce^{C\|\ue\|_{\lut\hs}}\Big(\|\vzeps\|_\hs + \|\aeps\nb\Pe\|_{\lutt\hs} + \eps\|\aeps\del\ue\|_{\lutt\hs}\Big).
\]
Owing to the tame estimate \eqref{eq:tame} and Proposition \ref{compo2}, we have 
\[
\|\aeps\nb\Pe\|_{\lutt\hs} \lesssim \B_{\eps,T}\|\nb\Pe\|_{\lutt\hs},
\]
and 
\[
\begin{split}
    \|\aeps\del\ue\|_{\lutt\hs} &\lesssim \|\ue\|_{\lutt\hspd} + \|\nb\aeps\|_{\litt\hsm}\|\ue\|_{\lut\hsp} \\
    &\lesssim \|\ue\|_{\lutt\hspd} + \|\nb\re\|_{\litt\hs}\|\ue\|_{\lutt H^{s+3/2}},
\end{split}
\]
where we have also used \eqref{timespace2} for the last inequality. Using \eqref{interpolation} then \eqref{timespace1}, we have 
\[
\|\ue\|_{\lutt H^{s+3/2}} \le \|\ue\|_{\lutt\hs}^{1/4}\|\ue\|_{\lutt\hspd}^{3/4} \le T^{1/4}\|\ue\|_{\litt\hs}^{1/4}\|\ue\|_{\lutt\hspd}^{3/4},
\]
so that
\[
\|\aeps\del\ue\|_{\lutt\hs} \lesssim T\|\nb\re\|_{\litt\hs}^4\|\ue\|_{\litt\hs} + \|\ue\|_{\lutt\hspd}.
\]
From these computations, we deduce that 
\[
\begin{split}
    \|\ve\|_{\litt\hs} \le Ce^{CT\|\ue\|_{\litt\hs}}\Big(&\|\vzeps\|_\hs + \B_{\eps,T}\|\nb\Pe\|_{\lutt\hs} \\ 
    &+ \eps T\|\nb\re\|_{\litt\hs}^4\|\ue\|_{\litt\hs} + \eps\|\ue\|_{\lutt\hspd}\Big).
\end{split}
\]
It follows that 
\[
\|\ve\|_{\litt\hs} \le Ce^{CT\eeps(T)}\Big(E_0 + (1+\eeps(T))\|\nb\Pe\|_{\lutt\hs} + \eps T\eeps(T)^5 + \eps\|\ue\|_{\lutt\hspd}\Big).
\]
From \eqref{ue} and \eqref{nbPe}, we then have 
\begin{align}
    \label{Ue}
    \|\ve\|_{\litt\hs} \le \;&Ce^{CT(1+\eeps(T))^\lambda\eeps(T)} \\ 
    \nonumber&\Big(E_0 + (1+\eeps(T))^{\lambda+1} \left[\left(T\eeps(T) + \eps(T^{1/2}+T)\right)\eeps(T) + T^{1/6}\eeps(T)\right] \\ 
    \nonumber&+\eps T\eeps(T)^5 + \eps(1+\eeps(T))^\lambda T\eeps(T)\Big).
\end{align}

Plugging now \eqref{ue} and \eqref{Ue} into \eqref{eeps}, we finally obtain that 
\begin{align}
    \label{eeps2}
    \eeps(T) \le \;&Ce^{CT(1+\eeps(T))^\lambda\eeps(T)} \\ 
    \nonumber&\Big(E_0 + (1+\eeps(T))^{\lambda+1} \left[\left(T\eeps(T) + (T^{1/2}+T)\right)\eeps(T) + T^{1/6}\eeps(T)\right] \\ 
    \nonumber&+ T\eeps(T)^5 + (1+\eeps(T))^\lambda T\eeps(T)\Big),
\end{align}
where we have also used the fact that $\eps\le1$. 

\medskip
Let us introduce the bounds
\begin{equation}
    \label{teps1}
    CT(1+\eeps(T))^\lambda\eeps(T)\le\log2,
\end{equation}
\begin{align}
    \label{teps2}
    &(1+\eeps(T))^{\lambda+1} \left[\left(T\eeps(T) + (T^{1/2}+T)\right)\eeps(T) + T^{1/6}\eeps(T)\right] \\
    \nonumber&+ T\eeps(T)^5 + (1+\eeps(T))^\lambda T\eeps(T)\le E_0.
\end{align}
We define the time 
\[
\teepse := \sup\Big\{0\le T\le\teeps : \eqref{teps1}-\eqref{teps2}\;{\rm are\;satisfied}\Big\}.
\]
Let us set $C_0:=4CE_0$. In view of \eqref{eeps2}, one has 
\begin{equation}
    \label{fepsteepse}
    \eeps(T) \le C_0,\qquad\forall\,0\le T\le\teepse.
\end{equation}
We then consider the bounds
\begin{equation}
    \label{teps1bis}
    CT(1+C_0)^\lambda C_0\le\log2,
\end{equation}
\begin{equation}
    \label{teps2bis}
    (1+C_0)^{\lambda+1} \left[\left(TC_0 + (T^{1/2}+T)\right)C_0 + T^{1/6}C_0\right] + TC_0^5 + (1+C_0)^\lambda TC_0\le E_0,
\end{equation}
and define the time
\[
T_0 := \sup\Big\{T\ge0 : \eqref{teps1bis}-\eqref{teps2bis}\;{\rm are\;satisfied}\Big\},
\]
which is independent of $0<\eps\le1$. By time continuity, at least one of the two conditions \eqref{teps1}-\eqref{teps2} becomes an equality at time $T=\teepse$: one has $C\teepse(1+\eeps(\teepse))^\lambda\eeps(\teepse)=\log2$, or  
\[
\begin{split}
    &(1+\eeps(\teepse))^{\lambda+1} \left[\left(\teepse\eeps(\teepse) + (\teepse^{1/2}+\teepse)\right)\eeps(\teepse) + \teepse^{1/6}\eeps(\teepse)\right] \\
    &+ \teepse\eeps(\teepse)^5 + (1+\eeps(\teepse))^\lambda\teepse\eeps(\teepse)=E_0.
\end{split}
\]
Using \eqref{fepsteepse}, one then has either $C\teepse(1+C_0)^\lambda C_0\ge\log2$, or  
\[
\begin{split}
    &(1+C_0)^{\lambda+1} \left[\left(\teepse C_0 + (\teepse^{1/2}+\teepse)\right)C_0 + \teepse^{1/6}C_0\right] \\
    &+ \teepse C_0^5 + (1+C_0)^\lambda\teepse C_0\ge E_0.
\end{split}
\]
By definition of $T_0$, we deduce that $\teepse\ge T_0$. By definition of $\teepse$, we also have that $\teeps\ge\teepse$. This implies that
\[
\inf_{0<\eps\le1} T_\eps\ge T_0,
\]
and that estimate \eqref{fepsteepse} holds true at time $T=T_0$:
\begin{equation}
    \label{eetz}
    \eeps(T_0)\le C_0.
\end{equation}
It remains to prove the uniform bound on $(\nb\Pe)_{0<\eps\le1}$. Let us define $\aeps:=1/\re$, which satisfies $\aeps\ge a_*$, where $a_*$ is defined by \eqref{anp}. From the second equation in \eqref{elseps} and the divergence-free condition on $\ue$, the pressure gradient $\nb\Pe$ satisfies the elliptic equation
\[
-\divv(\aeps\nb\Pe) = \divv\big((\ve\cdot\nb)\ue - \eps\aeps\del\ue\big).
\]

Using \eqref{eetz} together with Proposition \ref{compo2} and \eqref{timespace1}, we gather that the sequence $((\ve\cdot\nb)\ue)_{0<\eps\le1}$ is uniformly bounded in $L_{T_0}^\infty\hsm$.

Next, we know from \eqref{eetz} that the sequences $(\ue)_{0<\eps\le1}$ and $(\eps\ue)_{0<\eps\le1}$ are uniformly bounded, respectively, in $\widetilde{L_{T_0}^\infty}\hs$ and $\widetilde{L_{T_0}^1}\hspd$. By interpolation \eqref{interpolation}, we deduce that $(\eps^{1/2}\ue)_{0<\eps\le1}$ is uniformly bounded in $\widetilde{L_{T_0}^2}\hsp=L_{T_0}^2\hsp$, which in turn implies that $(\eps^{1/2}\del\ue)_{0<\eps\le1}$ is uniformly bounded in $L_{T_0}^2\hsm$. From Proposition \ref{compo2} and Corollary \ref{tameestimates}, this also implies that $(\eps\aeps\del\ue)_{0<\eps\le1}$ is uniformly bounded in $L_{T_0}^2\hsm$.

We finally gather from Proposition \ref{elliptichigh} that $(\nb\Pe)_{0<\eps\le1}$ is uniformly bounded in $L_{T_0}^2\hsm$. This completes the proof. 
\end{proof}

\section{Well-posedness for the original system}
\label{sec:proofth}

In this section, we conclude the proof of Theorem \ref{main}. It remains to show the convergence of the family $(\re,\ue,\nb\Pe)_{0<\eps\le1}$ towards a solution of system \eqref{odd3}, to prove the claimed regularity properties and energy equalities for this solution, and finally, to prove uniqueness.

\subsection{Proof of existence}
\label{sec:existence}

As a first step, we obtain the convergence of the family $(\re,\ue,\nb\Pe)_{0<\eps\le1}$, provided by Theorem \ref{th:oddeps}, towards a solution $(\rho,u,\nb\Pi)$ of system \eqref{odd3}. We then obtain a solution $(\rho,u,\nb\pi)$ to the original system \eqref{odd}, and prove the regularity properties and energy equalities stated in Theorem \ref{main}.

\medskip
As usual, we will simply denote $T\equiv T_0$, where $T_0$ is given by Theorem \ref{th:oddeps}.

\paragraph{Convergence of the sequence of regularised solutions.} For all $0<\eps\le1$, we have on $[0,T]\times\R^2$:
\begin{equation}
\label{oddeps2}
\left\{
\begin{aligned}
& (\dt + \ue\cdot\nb)\re = 0, \\
& \re(\dt + \ve\cdot\nb)\ue + \nb\Pe - \eps\del\ue = 0, \\
& \divv\,\ue = 0, \\
& (\re,\ue)_{|t=0} = (\rz,\uz),
\end{aligned}
\right.
\end{equation}
We aim to take the limit as $\eps\to0$ in some suitable functional space in the above system. 

From the first equation in \eqref{oddeps2}, 
\begin{equation}
    \label{rhonprz}
    \re(t)-\rz = -\int_0^t \ue\cdot\nb\re\dd\tau,\qquad\forall\,t\in[0,T].
\end{equation}
Using this, we deduce from \eqref{epsbound} and \eqref{timespace1} that the sequence $(\re-\rz)_{0<\eps\le1}$ is uniformly bounded in $\lit\hs$. We thus obtain the existence of some $r\in\lit\hs$ such that, up to extraction of a suitable subsequence, 
\[
\re-\rz\weaks r\qquad\mathrm{in}\;\lit\hs,\qquad\mathrm{as}\; \eps\to0.
\]
After defining $\rho:=r+\rz$, we then have 
\begin{equation}
    \label{cvw:rhon}
    \re-\rz\weaks\rho-\rz\qquad\mathrm{in}\;\lit\hs,\qquad\mathrm{as}\; \eps\to0.
\end{equation}
Using again the first equation in \eqref{oddeps2}, we have 
\[
\dt\re = -\ue\cdot\nb\re,
\]
so that, arguing as before, the sequence $(\pt_t\re)_{0<\eps\le1}$ is uniformly bounded in $\lit\hs$. The Aubin-Lions lemma \cite{Simon1987} then implies, up to extracting a suitable subsequence, the strong convergence 
\begin{equation}
    \label{cv:rhon}
    \re-\rz\to\rho-\rz\qquad\mathrm{in}\;\ct\hsig(B_R),\qquad\forall\,R>0,\qquad\forall\,\sig<s,\qquad\mathrm{as}\; \eps\to0,
\end{equation}
where we have also used \eqref{cvw:rhon} to deduce that the limit is indeed $\rho-\rz$, and performed a diagonal extraction to find a uniform subsequence in $R>0$. 

Finally, coming back to the uniform boundedness of $(\dt\re)_{0<\eps\le1}$ in $\lit\hs$, we deduce that 
\begin{equation}
    \label{cv:dtrhon}
    \dt\re\weaks\dt\rho\qquad\mathrm{in}\;\lit\hs,\qquad\mathrm{as}\;\eps\to0.
\end{equation}

\medskip
We now derive similar convergence properties for $(\ue)_{0<\eps\le1}$. From \eqref{epsbound} and \eqref{timespace1}, we gather the existence of some $u\in\lit\hs$ such that, up to extraction of a suitable subsequence, 
\begin{equation}
    \label{cvw:un}
    \ue\weaks u\qquad\mathrm{in}\;\lit\hs,\qquad\mathrm{as}\;\eps\to0.
\end{equation}
Consider now the second equation in \eqref{oddeps2}. We have 
\[
\pt_t\ue = - (\ve\cdot\nb)\ue - \aeps\nb\Pe + \eps\aeps\del\ue,
\]
where $\aeps:=1/\re$.

We know from the proof of Proposition \ref{epsbounds} that $\left((\ve\cdot\nb)\ue\right)_{0<\eps\le1}$ is uniformly bounded in $\lit\hsm$, and $(\eps\aeps\del\ue)_{0<\eps\le1}$ is uniformly bounded in $\ldt\hsm$. Moreover, $(\nb\Pe)_{0<\eps\le1}$ is uniformly bounded in $\ldt\hsm$, and so is $(\aeps\nb\Pe)_{0<\eps\le1}$. 

With all these bounds, we gather that $(\pt_t\ue)_{0<\eps\le1}$ is uniformly bounded in $\ldt\hsm$. As also $(\ue)_{0<\eps\le1}$ is uniformly bounded in $\lit\hs$, we can make use of the convergence property \eqref{cvw:un} and argue as for the sequence $(\re-\rz)_{0<\eps\le1}$ above to obtain, up to extraction of a suitable subsequence, 
\begin{equation}
    \label{cv:un}
    \ue\to u\qquad\mathrm{in}\;\ct\hsig(B_R),\qquad\forall\,R>0,\qquad\forall\,\sig<s,\qquad\mathrm{as}\;\eps\to0.
\end{equation}

Using again that $(\pt_t\ue)_{0<\eps\le1}$ is uniformly bounded in $\ldt\hsm$, we have
\begin{equation}
    \label{cv:dtun}
    \dt\ue\weak\dt u\qquad\mathrm{in}\;\ldt\hsm,\qquad\mathrm{as}\;\eps\to0.
\end{equation}
Arguing as before, we also gather that $\divv\,u=0$ on $[0,T]\times B_R$, and that $(\ve)_{0<\eps\le1}$ satisfies the strong convergence 
\begin{equation}
    \label{cv:Un}
    \ve\to U\qquad\mathrm{in}\;\ct H^\sig(B_R),\qquad\forall\,R>0,\qquad\forall\,\sig<s,\qquad\mathrm{as}\;\eps\to0,
\end{equation}
where we have defined the divergence-free vector field
\[
U := u - \nbp g(\rho).
\]
From the proof of Proposition \ref{epsbounds}, we know that $(\eps^{1/2}\del\ue)_{0<\eps\le1}$ is uniformly bounded in $\ldt\hsm$, which implies 
\begin{equation}
    \label{cv:epsdel}
    \eps\del\ue\to0\qquad\mathrm{in}\;\ldt\hsm,\qquad\mathrm{as}\;\eps\to0.
\end{equation}
We also know that $(\nb\Pe)_{0<\eps\le1}$ is uniformly bounded in $\ldt\hsm$: there exists $\nb\Pi\in\ldt\hsm$ such that 
\begin{equation}
    \label{cv:nbPin}
    \nb\Pe\weak\nb\Pi\qquad\mathrm{in}\;\ldt\hsm,\qquad\mathrm{as}\;\eps\to0.
\end{equation}
In view of the convergence properties \eqref{cv:rhon}, \eqref{cv:dtrhon}, \eqref{cv:un}, \eqref{cv:dtun}, \eqref{cv:Un}, \eqref{cv:epsdel} and \eqref{cv:nbPin}, we can take the limit as $\eps\to0$ in the weak formulation of system \eqref{oddepsbis} on $[0,T]\times B_R$, to gather that the triplet $(\rho,u,\nb\Pi)$ solves on $[0,T]\times B_R$ the system
\begin{equation}
    \label{odd2bis}
    \left\{
    \begin{aligned}
    & (\dt + u\cdot\nb)\rho = 0, \\
    & \rho(\dt + U\cdot\nb)u + \nabla\Pi = 0, \\
    & \divv\,u = 0, \\
    & (\rho,u)_{|t=0} = (\rz,\uz).
\end{aligned}
\right.
\end{equation}
This being valid for any $R>0$, we finally obtain that $(\rho,u,\nb\Pi)$ is a solution on $[0,T]\times\R^2$ to system \eqref{odd2bis}. From the first equation, we also deduce that $\rho_*\le\rho\le\rho^*$ on $[0,T]\times\R^2$.

\medskip
In view of the above properties, we can reproduce the computations of Subsection \ref{sec:reformulation} to gather that the quadruple $(\rho,u,U,\nb\Pi)$ solves the Elsässer formulation \eqref{els}.

\medskip
Now, define 
\begin{equation}
    \label{eq:pi}
    \pi := \Pi + f(\rho)\om,
\end{equation}
with $\om:=\curl(u)=\pt_1u_2-\pt_2u_1$. Performing the computations of Subsection \ref{sec:reformulation} backwards, we immediately gather that the triplet $(\rho,u,\nb\pi)$ is a solution of the original system \eqref{odd}.

\paragraph{Regularity properties.} Let us rewrite the Elsässer formulation \eqref{els} on $[0,T]\times\R^2$ as
\begin{equation}
    \label{odd2els}
    \left\{
    \begin{aligned}
    & (\dt + u\cdot\nb)\rho = 0, \\
    & (\dt + U\cdot\nb)\rho = 0, \\
    & (\dt + U\cdot\nb)u + a\nb\Pi = 0, \\
    & (\dt + u\cdot\nb)U + a\nb\Pi = 0, \\
    & \divv\,u = \divv\,U = 0, 
\end{aligned}
\right.
\end{equation}
where we have defined $a:=1/\rho$. Recall that, at this point, we only have the regularity properties $(\nb\rho,u,U)\in\left(\lit\hs\right)^3$ and $\nb\Pi\in\ldt\hsm$.

\medskip
To begin with, let us investigate the regularity of the pressure gradient, which satisfies the elliptic equation 
\begin{equation}
    \label{anbPi}
    -\divv(a\nb\Pi) = \divv\big((U\cdot\nb)u\big).
\end{equation}
From the above regularity properties, we have $u$, $U\in\lit\hs$. Using also the ellipticity property $a\ge a_*>0$, we can apply Proposition \ref{elliptichigh} to obtain that $\nb\Pi\in\lit\hs$. Notice that from the usual tame estimates, we also have $a\nb\Pi\in\lit\hs$.

Next, we turn our attention to the velocity field $u$. Let us rewrite the third equation in \eqref{odd2els} as 
\[
(\dt + U\cdot\nb)u = -a\nb\Pi.
\]
Since the right-hand side and the transport field $U$ belong to $\lit\hs$, and the initial datum $u_0$ belongs to $\hs$, we can apply Theorem \ref{th:transport2} to gather that $u\in\ct\hs$.

Arguing in the same way for the fourth equation in \eqref{odd2els}, we find that $U\in\ct\hs$. Since we also have $u\in\ct\hs$, we deduce that $\nbp g(\rho)\in\ct\hs$. From the assumptions on $g$, we finally obtain that $\nb\rho\in\ct\hs$.

From these new regularity properties for $u$ and $U$, we deduce from \eqref{anbPi} and the classical theory for elliptic equations that $\nb\Pi\in\ct\hs$.

Let us now complete the regularity properties for the density $\rho$. From the first equation in \eqref{odd2els}, we have 
\[
\rho-\rz = -\int_0^t u\cdot\nb\rho,
\]
from which we deduce that $\rho-\rz\in\ct\hs$. We also have 
\[
\pt_t(\rho-\rz) = \pt_t\rho = - u\cdot\nb\rho,
\]
so that $\pt_t(\rho-\rz)\in\ct\hs$. We thus have $\rho-\rz\in\cut\hs$. From this, we also deduce that $\nb(\rho-\rz)\in\cut\hsm$. Since we have $\nb\rz\in\hs$, it follows that $\nb\rho\in\cut\hsm$.

Now, since $\rho_*\le\rho\le\rho^*$, we deduce from the first equation in \eqref{odd2els} that 
\[
(\dt + u\cdot\nb)a = 0.
\]
Arguing as for $\rho$, we then have that $a-a_0\in\cut\hs$. From this and the fact that $\nb\Pi\in\ct\hs$, we deduce that $(a-a_0)\nb\Pi\in\ct\hs$. Using again that $\nb\Pi\in\ct\hs$ and the classical tame estimates, we have $a_0\nb\Pi\in\ct\hs$, so that finally $a\nb\Pi\in\ct\hs$.

Since 
\[
\pt_tu = -(U\cdot\nb)u - a\nb\Pi,
\]
we deduce that $\pt_tu\in\ct\hsm$. We thus have that $u\in\cut\hsm$.

It remains to verify the claimed regularity for $\nb\pi$. From \eqref{eq:pi} and \eqref{anbPi}, we see that $\nb\pi$ satisfies the elliptic equation 
\[
-\divv(a\nb\pi) = \divv\big((U\cdot\nb)u - a\nb(f(\rho)\om)\big).
\]
Next, using the standard tame and paralinearisation estimates, we have 
\[
\|\nb(f(\rho)\om)\|_\hsmd \le \|f(\rho)\om\|_\hsm \lesssim \big(\|f\|_{\li([\rho_*,\rho^*])} + \|\nb\rho\|_\hsmd\big)\|u\|_\hs,
\]
so that $\nb(f(\rho)\om)\in\lit\hsmd$, which in turn implies, using again the tame and paralinearisation estimates, that $a\nb(f(\rho)\om)\in\lit\hsmd$. From this and the previous regularity properties, we deduce from classical elliptic theory that $\nb\pi\in\ct\hsmd$. This completes the proof of the regularity properties of Theorem \ref{main}.

\paragraph{Energy equalities.} Let us write the third equation in \eqref{odd2els} as 
\begin{equation}
    \label{odd2bis2}
    \rho(\dt + U\cdot\nb)u + \nabla\Pi = 0.
\end{equation}
From the above regularity properties, we can write
\[
\int\rho\pt_tu\cdot u \dd x= \frac{1}{2}\ddt\|\sqrt{\rho}u\|_\ld^2 - \frac{1}{2}\int\pt_t\rho|u|^2\dd x,\qquad \int\rho(U\cdot\nb)u\cdot u \dd x= -\frac{1}{2}\int(U\cdot\nb\rho)|u|^2 \dd x,
\]
so that
\[
\int\rho\big(\dt + U\cdot\nb\big)u\cdot u \dd x = \frac{1}{2}\ddt\|\sqrt{\rho}u\|_\ld^2 - \frac{1}{2}\int\big(\dt+U\cdot\nb\big)\rho\cdot|u|^2 \dd x = \frac{1}{2}\ddt\|\sqrt{\rho}u\|_\ld^2,
\]
where we have also used the second equation in \eqref{odd2els} and the divergence-free condition on $U$. Since $u$ is divergence-free, it holds that
\[
\int \nb\Pi\cdot u \dd x = 0.
\]
We can thus perform an energy estimate on equation \eqref{odd2bis2} to gather that 
\[
\frac{1}{2}\ddt\|\sqrt{\rho}u\|_\ld^2 = 0,
\]
which yields the claimed energy equality for $(\rho,u)$.

\medskip
In view of the Elsässer formulation \eqref{odd2els}, we can switch the roles of $u$ and $U$ in the above computations, to gather the desired energy equality for $(\rho,U)$.

\medskip
The proof of the existence statement of Theorem \ref{main} is now complete.

\subsection{Proof of uniqueness}
\label{sec:uniqueness}

We establish a stability criterion for solutions to the Elsässer formulation \eqref{els}, which implies the uniqueness statement of Theorem \ref{main}.

\medskip
The following result is inspired from \cite[Proposition 5.3]{fv}.

\begin{proposition}
    \label{prop:stab}
    Let $T>0$. Assume that we dispose of two quadruples $(\rho^{(1)},u^{(1)},U^{(1)},\nb\Pi^{(1)})$ and $(\rho^{(2)},u^{(2)},U^{(2)},\nb\Pi^{(2)})$ of solutions on $[0,T]\times\R^2$ to the Elsässer formulation \eqref{els} 
    \begin{equation}
    \label{oddelsstab}
    \left\{
    \begin{aligned}
    & (\dt + u\cdot\nb)\rho = 0, \\
    & (\dt + U\cdot\nb)\rho = 0, \\
    & \rho(\dt + U\cdot\nb)u + \nb\Pi = 0, \\
    & \rho(\dt + u\cdot\nb)U + \nb\Pi = 0, \\
    & \divv\,u = \divv\,U = 0.
    \end{aligned}
    \right.
    \end{equation}
    Suppose that:
    \begin{itemize}
        \item $\rho^{(j)}(t,x)\in[\rho_*,\rho^*]$, for some $0<\rho_*<\rho^*<\infty$ and all $(t,x)\in[0,T]\times\R^2$ and $j=1,2$;
        \item $\nb\rho^{(2)}$, $\nb u^{(2)}$, $\nb U^{(2)}$, $\nb\Pi^{(2)}$ belong to $L^1([0,T];\li(\R^2))$.
    \end{itemize}
    For $h\in\{\rho,u,U,\nb\Pi\}$, set $\delta h:=h^{(1)}-h^{(2)}$, and define for all $t\in[0,T]$ the energy
    \[
    \D(t) := \left\|\left(\delta\rho,\delta u,\delta U\right)(t)\right\|_\ld^2.
    \]
    Then, there exists a constant $C=C(\rho_*,\rho^*)>0$, depending only on the quantities inside the brackets, such that 
    \[
    \sup_{t\in[0,T]}\D(t) \le Ce^{C\int_0^T I(t)\dd t}\D(0),
    \]
    where $I\in L^1([0,T])$ is defined by 
    \begin{equation}
        \label{eq:it}
        I(t) := \|\nb\rho^{(2)}(t)\|_\li + \|\nb u^{(2)}(t)\|_\li + \|\nb U^{(2)}(t)\|_\li + \|\nb\Pi^{(2)}(t)\|_\li,\qquad\forall\,t\in[0,T].
    \end{equation}
\end{proposition}

\begin{proof}
    Let us start by estimating $\delta\rho$. Writing the first equation in \eqref{oddelsstab} for both $(\rho^{(1)},u^{(1)})$ and $(\rho^{(2)},u^{(2)})$, and taking the difference, we find that $\delta\rho$ satisfies the transport equation
    \[
    \big(\dt + u^{(1)}\cdot\nb\big)(\delta\rho) = -\delta u\cdot\nb\rho^{(2)}.
    \]
    Multiplying this equation by $\delta\rho$ and integrating by parts, we get
    \[
    \frac{1}{2}\ddt\|\delta\rho\|_\ld^2 \le \|\delta\rho\|_\ld \|\delta u\|_\ld \|\nb\rho^{(2)}\|_\li \lesssim \left(\|\delta\rho\|_\ld^2 + \|\delta u\|_\ld^2\right)\|\nb\rho^{(2)}\|_\li,
    \]
    where we have also used the divergence-free condition on $u^{(1)}$. We then obtain that 
    \begin{equation}
        \label{eq:deltrho}
        \|\delta\rho(t)\|_\ld^2 \lesssim \|\delta\rho(0)\|_\ld^2 + \int_0^t \|\nb\rho^{(2)}(\tau)\|_\li \D(\tau)\dd\tau,\qquad \forall\,t\in[0,T].
    \end{equation}
    
    \medskip
    We now estimate $\delta u$. Let us reformulate the third equation in \eqref{oddelsstab} as 
    \[
    (\dt + U\cdot\nb)u + \frac{1}{\rho}\nb\Pi = 0.
    \]
    Writing this equation for both $(\rho^{(1)},u^{(1)},U^{(1)},\nb\Pi^{(1)})$ and $(\rho^{(2)},u^{(2)},U^{(2)},\nb\Pi^{(2)})$, taking the difference, and multiplying the resulting expression by $\rho^{(1)}$, we find that $\delta u$ satisfies the transport equation
    \begin{equation}
        \label{nbPidiff}
        \rho^{(1)}\big(\dt + U^{(1)}\cdot\nb\big)(\delta u) + \delta(\nb\Pi) = -\rho^{(1)}(\delta U\cdot\nb)u^{(2)} + \frac{\delta\rho}{\rho^{(2)}}\nb\Pi^{(2)}.
    \end{equation}
    Since
    \begin{align*}
        &\int_{\R^2}\rho^{(1)}\pt_t(\delta u)\cdot (\delta u) \dd x = \frac{1}{2}\ddt\|\sqrt{\rho^{(1)}}(\delta u)\|_\ld^2 - \frac{1}{2}\int_{\R^2}\pt_t\rho^{(1)}|\delta u|^2\dd x, \\
        &\int_{\R^2}\rho(U^{(1)}\cdot\nb)(\delta u)\cdot (\delta u)\dd x = -\frac{1}{2}\int_{\R^2}(U^{(1)}\cdot\nb\rho^{(1)})|\delta u|^2\dd x,
    \end{align*}
    we infer that
    \begin{align*}
        \int_{\R^2}\rho^{(1)}\left(\pt_t + U^{(1)}\cdot\nb\right)(\delta u)\cdot (\delta u)\dd x &= \frac{1}{2}\ddt\|\sqrt{\rho^{(1)}}(\delta u)\|_\ld^2 - \frac{1}{2}\int_{\R^2}\left(\dt+U^{(1)}\cdot\nb\right)\rho^{(1)}\cdot|\delta u|^2\dd x \\
        &= \frac{1}{2}\ddt\|\sqrt{\rho^{(1)}}(\delta u)\|_\ld^2,
    \end{align*}
    where we have also used the second equation in \eqref{oddelsstab} and the divergence-free condition on $U^{(1)}$. Since $\delta u$ is divergence-free, it holds that
    \[
    \int_{\R^2} \delta(\nb\Pi)\cdot (\delta u)\dd x = 0.
    \]
    We can thus perform an energy estimate to gather that 
    \begin{align*}
        \frac{1}{2}\ddt\|\sqrt{\rho^{(1)}}(\delta u)\|_\ld^2 &\lesssim \|\delta u\|_\ld\left(\|\delta U\|_\ld\|\nb u^{(2)}\|_\li + \|\delta\rho\|_\ld\|\nb\Pi^{(2)}\|_\li\right) \\
        &\lesssim \left(\|\delta\rho\|_\ld^2 + \|\delta u\|_\ld^2 + \|\delta U\|_\ld^2\right)\left(\|\nb u^{(2)}\|_\li + \|\nb\Pi^{(2)}\|_\li\right),
    \end{align*}
    for an implicit constant depending only on $\rho_*$ and $\rho^*$. From this and the fact that $\rho^{(1)}\ge\rho_*$, we obtain 
    \begin{equation}
        \label{eq:deltu}
        \|\delta u(t)\|_\ld^2 \lesssim \|\delta u(0)\|_\ld^2 + \int_0^t \left(\|\nb u^{(2)}(\tau)\|_\li + \|\nb\Pi^{(2)}(\tau)\|_\li\right) \D(\tau)\dd\tau,\qquad \forall\,t\in[0,T].
    \end{equation}
    
    \medskip
    It remains to estimate $\delta U$. The roles of the variables $u$ and $U$ being exactly symmetrical in system \eqref{oddelsstab}, one can perform the same computations as for $\delta u$ to deduce that
    \begin{equation}
        \label{eq:deltU}
        \|\delta U(t)\|_\ld^2 \lesssim \|\delta U(0)\|_\ld^2 + \int_0^t \left(\|\nb U^{(2)}(\tau)\|_\li + \|\nb\Pi^{(2)}(\tau)\|_\li\right) \D(\tau)\dd\tau,\qquad \forall\,t\in[0,T].
    \end{equation}
    
    \medskip
    Summing up estimates \eqref{eq:deltrho}, \eqref{eq:deltu}, \eqref{eq:deltU}, we gather that 
    \[
    \D(t) \lesssim \D(0) + \int_0^t I(\tau)\D(\tau)\dd\tau,
    \]
    where $I(t)$ is defined by \eqref{eq:it}. An application of the Grönwall lemma finally yields the desired inequality.
\end{proof} 

We are now ready to prove the uniqueness statement of Theorem \ref{main}.

\medskip\noindent
Let $(\rho^{(1)},u^{(1)},\nb\pi^{(1)})$ and $(\rho^{(2)},u^{(2)},\nb\pi^{(2)})$ be two solutions of system \eqref{odd} on $[0,T]\times\R^2$, for some $0<T<\infty$, stemming from the same initial data $(\rho_0,u_0)$ satisfying the conditions of Theorem \ref{main}. We define 
\[
U^{(j)} := u^{(j)} - \nbp g(\rho^{(j)}),\qquad \Pi^{(j)} := \pi^{(j)} - f(\rho^{(j)})\om^{(j)},\qquad j=1,2,
\]
where $\om^{(j)}:=\curl(u^{(j)})$. Then, the quadruples $(\rho^{(1)},u^{(1)},U^{(1)},\nb\Pi^{(1)})$ and $(\rho^{(2)},u^{(2)},U^{(2)},\nb\Pi^{(2)})$ solve the Elsässer formulation \eqref{oddelsstab}, and satisfy the boundedness and regularity properties of Theorem \ref{main}. In particular, the first condition of Proposition \ref{prop:stab} is satisfied, and, since $T<\infty$, we immediately have that $\nb\rho^{(2)}$, $\nb u^{(2)}$ and $\nb\Pi^{(2)}$ belong to $L^1([0,T];\li(\R^2))$. Moreover, as 
\[
\|\nb^2g(\rho^{(2)})\|_\li \lesssim \|g''(\rho^{(2)})\|_\li \|\nb\rho^{(2)}\|_\li^2 + \|g'(\rho^{(2)})\|_\li \|\nb^2\rho^{(2)}\|_\li,
\]
we also have $\nb U^{(2)}\in L^1([0,T];\li(\R^2))$. We can thus apply Proposition \ref{prop:stab} to gather that 
\[
\rho^{(1)} = \rho^{(2)}\qquad\mathrm{and}\qquad u^{(1)} = u^{(2)}\qquad\mathrm{in}\;L^\infty([0,T];\ld(\R^2)),
\]
thus almost everywhere on $[0,T]\times\R^2$. As these functions are continuous on $[0,T]\times\R^2$, these equalities are valid everywhere on $[0,T]\times\R^2$. 

From equation \eqref{nbPidiff}, we see that $\delta(\nb\Pi)$ satisfies the elliptic equation 
\[
-\divv\Big(\frac{1}{\rho^{(1)}}\delta(\nb\Pi)\Big) = \divv\Big(-\frac{\delta\rho}{\rho^{(1)}\rho^{(2)}}\nb\Pi^{(2)} + \left(U^{(1)}\cdot\nb\right)(\delta u) + (\delta U\cdot\nb)u^{(2)}\Big) = 0.
\]
We can thus apply Proposition \ref{elliptichigh} to gather that $\delta(\nb\Pi)=0$ in $C([0,T];\hs(\R^2))$, hence $\nb\Pi^{(1)}=\nb\Pi^{(2)}$ everywhere on $[0,T]\times\R^2$, as $s>2$.

From the above properties, we finally deduce that $\nb\pi^{(1)}=\nb\pi^{(2)}$ everywhere on $[0,T]\times\R^2$.

\medskip
The proof of the uniqueness statement of Theorem \ref{main} is now complete.

\appendix

\section{Analysis of a Navier--Stokes type system}
\label{sec:stokes}

We consider the system           
\begin{equation}
\label{stokes}
\left\{
\begin{aligned}
& (\dt + v\cdot\nb)u + a\nb\Pi - \nu a\del u = 0, \\
& \divv\,u = 0, \\
& u_{|t=0} = u_0,
\end{aligned}
\right.
\end{equation}
where the functions $u_0$, $a$, $v$, and the viscosity parameter $\nu>0$ are \emph{fixed} given data. We are interested in solving \eqref{stokes} in the variables $(u,\nb\Pi)$. 

\medskip
The following result is an adaptation of \cite[Propositions 3.2 and 3.4]{danchin2006}, where we provide new estimates for the couple $(u,\nb\Pi)$.
\begin{theorem}
\label{thstokes}
    Let $T>0$, $s>2$ and $\nu>0$. Let $u_0\in\hs$ be such that $\divv\,u_0=0$. Let $(a,v)\in\li([0,T]\times\R^2)\times\left(\litt\hs\cap\lut\hs\right)(\R^2)$ be such that 
    \[
    0<a_*:=\inf_{(t,x)\in[0,T]\times\R^2}a(t,x)\le a^*:=\|a\|_\li,\qquad \nb a\in\litt\hs(\R^2), \qquad\divv\,v=0.
    \]
    Then, there exists a unique solution $(u,\nb\Pi)$ to system \eqref{stokes} such that 
    \[
    u\in\ctt\hs(\R^2),\qquad \nu u\in\lutt\hspd(\R^2),\qquad \nb\Pi\in\lutt\hs(\R^2).
    \]
    Moreover, after defining
    \begin{equation}
        \label{at}
        \A_T := 1 + \|\nb a\|_{\litt\hs},
    \end{equation}
    there exist constants $C=C(s,a_*,a^*)>0$ and $\lambda=\lambda(s)>0$, depending only on the quantities inside the brackets, such that the solution $(u,\nb\Pi)$ satisfies the estimates
    \[
    \|u\|_{\litt\hs} + \nu\|u\|_{\lutt\hspd} \le Ce^{C\A_T^\lambda\|v\|_{\lut\hs}} \left(\|u_0\|_\hs + \nu T\A_T^\lambda\|u\|_{\litt\hs}\right),
    \]
    \[
    \|\nb\Pi\|_{\lutt\hs} \le C\A_T^\lambda \left(\left(T\|v\|_{\litt\hs} + \nu\left(T^{1/2}+T\right)\right)\|u\|_{\litt\hs} + \nu T^{1/6}\|u\|_{\lutt\hspd}\right).
    \]
\end{theorem}

\medskip
Theorem \ref{thstokes} is needed at two levels of the proof of Theorem \ref{main}. Firstly, we use it in Subsection \ref{subsec:approx} to construct uniformly bounded approximate solutions to the regularised system \eqref{oddeps}. For this, we need the last term in the bracket in the estimate for $u$ to have a power of $T$ as a factor, in order to deduce a uniform bound from the energy inequality \eqref{eq:energ1}. We also use it to derive uniform bounds for the solutions to the regularised system \eqref{oddeps}. In this step, one additionally needs all the terms in the estimate for the pressure to have a power of $T$ as a factor, in order to perform a continuation argument from the energy inequality \eqref{eeps2} to gather a uniform time of existence for the family of regularised solutions. These are the reasons why we provide the modified estimates above. Finally, let us mention that the uniqueness statement plays no role in the proof of Theorem \ref{main}, as only the existence and estimates are needed in Section \ref{sec:reg}. 

\medskip
We are now ready to get into the proof. The existence and uniqueness statements being already proved in \cite[Proposition 3.4]{danchin2006}, we focus only on the proof of the estimates. As they cannot be deduced from those stated in \cite[Proposition 3.2]{danchin2006}, we need to prove them directly from system \eqref{stokes}.

\begin{proof}
To begin with, let us notice that, for all $j\ge-1$, we have 
\[
\begin{split}
    \del_j(a\del u) &= a\del\del_ju + [\del_j,a]\del u \\
    &= \divv(a\nb\del_ju) - (\nb a\cdot\nb)\del_ju + [\del_j,a]\del u.
\end{split}
\]
We can thus apply the operator $\del_j$ to the first equation in \eqref{stokes} to gather that 
\[
\dt\del_ju + (v\cdot\nb)\del_ju - \nu\,\divv(a\nb\del_ju) = -\del_j(a\nb\Pi) - \nu(\nb a\cdot\nb)\del_ju + C_j,
\]
where we have defined the commutator term 
\[
C_j := [v\cdot\nb,\del_j]u + \nu[\del_j,a]\del u.
\]
Multiplying this equation by $\del_ju$ and integrating by parts, we find that 
\[
\frac{1}{2}\ddt\|\del_ju\|_\ld^2 + \nu\|\nb\del_j u\|_\ld^2 \lesssim \|\del_ju\|_\ld \Big(\|\del_j(a\nb\Pi)\|_\ld + \nu\|\nb a\|_\li\|\nb\del_j u\|_\ld + \|C_j\|_\ld\Big),
\]
for an implicit constant depending only on $a_*$. From the Bernstein inequalities \eqref{Bernstein}, we deduce that, for all $j\ge-1$, we have 
\[
\frac{1}{2}\ddt\|\del_ju\|_\ld^2 + \nu\|\nb\del_j u\|_\ld^2 \lesssim \|\del_ju\|_\ld \Big(\|\del_j(a\nb\Pi)\|_\ld + \nu2^j\|\nb a\|_\li\|\del_j u\|_\ld + \|C_j\|_\ld\Big).
\]
For $j\ge0$, we gather from the second Bernstein inequality in \eqref{Bernstein} that 
\[
\ddt\|\del_ju\|_\ld + \nu2^{2j}\|\del_ju\|_\ld \lesssim \|\del_j(a\nb\Pi)\|_\ld + \nu2^j\|\nb a\|_\li\|\del_j u\|_\ld + \|C_j\|_\ld.
\]
As for $j=-1$, we have 
\[
\ddt\|\del_{-1}u\|_\ld + \nu2^{-2}\|\del_{-1}u\|_\ld \lesssim \nu2^{-2}\|\del_{-1}u\|_\ld + \|\del_{-1}(a\nb\Pi)\|_\ld + \nu2^{-1}\|\nb a\|_\li\|\del_{-1}u\|_\ld + \|C_{-1}\|_\ld.
\]
Integrating these last two inequalities on $[0,T]$, multiplying the resulting expressions by $2^{js}$ and performing an $\ell^2$ summation over $j\ge-1$, we obtain 
\begin{align}
    \label{stokesbound}
    \|u\|_{\litt\hs} + \nu\|u\|_{\lutt\hspd} \lesssim\; &\|u_0\|_\hs + \nu2^{-2}\|\del_{-1}u\|_{\lut\ld} + \|a\nb\Pi\|_{\lutt\hs} \\
    \nonumber&+ \nu\|\nb a\|_{\lit\li}\|u\|_{\lutt\hsp} + \left\|\left(2^{js}\|C_j\|_{\lut\ld}\right)_{j\ge-1}\right\|_{\ell^2}.
\end{align}
We now estimate the terms on the right-hand side. 

\medskip
First of all, using \eqref{deltcont} and \eqref{timespace1}, we have 
\begin{equation}
    \label{stokesbound1}
    \|\del_{-1}u\|_{\lut\ld} \le T\|u\|_{\litt\hs}.
\end{equation}
Next, using the interpolation inequality \eqref{interpolation}, it follows that 
\begin{equation}
    \label{stokesbound2}
    \nu\|\nb a\|_{\lit\li}\|u\|_{\lutt\hsp} \le \nu T^{1/2}\A_T\|u\|_{\litt\hs}^{1/2} \|u\|_{\lutt\hspd}^{1/2}.
\end{equation}
Let us now estimate the commutator term $C_j$. Using the first item of Proposition \ref{commut}, we infer that 
\[
\Big\|\Big(2^{js}\big\|[\del_j,a]\del u\big\|_{\lut\ld}\Big)_{j\ge-1}\Big\|_{\ell^2} \lesssim \|\nb a\|_{\litt\hs}\|u\|_{\lutt H^{s+3/2}} \lesssim \|\nb a\|_{\litt\hs}T^{1/4}\|u\|_{\litt\hs}^{1/4}\|u\|_{\lutt\hspd}^{3/4}.
\]
From the second item of Proposition \ref{commut}, we have 
\[
\Big\|\Big(2^{js}\big\|[v\cdot\nb,\del_j]u\big\|_{\lut\ld}\Big)_{j\ge-1}\Big\|_{\ell^2} \lesssim \int_0^T \|\nb v(t)\|_\hsm \|u\|_{\widetilde{L_t^\infty}\hs} \dd t.
\]
In view of these two estimates, we obtain 
\begin{equation}
    \label{stokesbound3}
    \Big\|\Big(2^{js}\|C_j\|_{\lut\ld}\Big)_{j\ge-1}\Big\|_{\ell^2} \lesssim \int_0^T \|\nb v(t)\|_\hsm \|u\|_{\widetilde{L_t^\infty}\hs} \dd t + \nu T^{1/4}\A_T\|u\|_{\litt\hs}^{1/4}\|u\|_{\lutt\hspd}^{3/4}.
\end{equation}
We now bound the pressure gradient, which satisfies the elliptic equation
\[
-\divv(a\nb\Pi) = \divv(-\nu a\del u + (v\cdot\nb)u).
\]
From this equation and Proposition \ref{compo2}, we can apply Proposition \ref{elliptichigh} to gather the bound 
\begin{equation}
    \label{estnbPi}
    \|\nb\Pi\|_\hs \lesssim \nu\|\nb a\|_\hsm\|u\|_\hsp + \|v\|_\hs\|u\|_\hs + \big(1+\|\nb a\|_\hsm\big)^\gamma\big(\nu\|u\|_{H^2} + \|v\|_\hsm\|u\|_{H^1}\big)
\end{equation}
for some constant $\gamma=\gamma(s)>0$, where we have used the fact that $a\le a^*$. Integrating in time and using the bounds \eqref{timespace1}-\eqref{timespace2}, one gathers that 
\[
\begin{split}
    \|\nb\Pi\|_{\lutt\hs} \lesssim\; &\nu\A_T\|u\|_{\lutt H^{s+3/2}} + \int_0^T \|v(t)\|_\hs\|u\|_{\widetilde{L_t^\infty}\hs} \dd t \\
    &+ \A_T^\gamma\Big(\nu\|u\|_{\lutt\hs} + \int_0^T \|v(t)\|_\hs\|u\|_{\widetilde{L_t^\infty}\hs} \dd t\Big),
\end{split}
\]
where $\A_T$ is defined by \eqref{at}. We deduce that
\begin{equation}
    \label{stokesPi}
    \|\nb\Pi\|_{\lutt\hs} \lesssim \A_T^{\gamma+1}\Big(\int_0^T \|v(t)\|_\hs\|u\|_{\widetilde{L_t^\infty}\hs} \dd t + \nu T\|u\|_{\litt\hs} + \nu T^{1/4}\|u\|_{\litt\hs}^{1/4}\|u\|_{\lutt\hspd}^{3/4}\Big),
\end{equation}
which in turn implies from \eqref{eq:tame} that 
\begin{equation}
    \label{stokesbound4}
    \|a\nb\Pi\|_{\lutt\hs} \lesssim \A_T^{\gamma+2}\Big(\int_0^T \|v(t)\|_\hs\|u\|_{\widetilde{L_t^\infty}\hs} \dd t + \nu T\|u\|_{\litt\hs} + \nu T^{1/4}\|u\|_{\litt\hs}^{1/4}\|u\|_{\lutt\hspd}^{3/4}\Big).
\end{equation}
Plugging estimates \eqref{stokesbound1}, \eqref{stokesbound2}, \eqref{stokesbound3} and \eqref{stokesbound4} into \eqref{stokesbound}, using the Young inequality in the last term to absorb $\|u\|_{\lutt\hspd}$ on the left-hand side, and finally applying the Grönwall lemma, we find the desired estimate for $u$.

\medskip
Finally, from \eqref{stokesPi}, we deduce that
\[
\|\nb\Pi\|_{\lutt\hs} \lesssim \A_T^{\gamma+1}\left(T\|v\|_{\litt\hs}\|u\|_{\litt\hs} + \nu T\|u\|_{\litt\hs} + \nu T^{1/4}\|u\|_{\litt\hs}^{1/4}\|u\|_{\lutt\hspd}^{3/4}\right).
\]
Using the Young inequality in the last term then yields the claimed estimate for the pressure.
\end{proof}

\section{Elements of Littlewood-Paley theory}
\label{sec:tools}

In this appendix, we collect all the results from Littlewood-Paley theory needed for our study. The majority of the results presented here are borrowed from \cite[Chapters 2 and 3]{bcd}. As the dimension plays no role, we will work in $\R^d$, with $d\ge1$.

\subsection{Littlewood-Paley decomposition and Besov spaces}

Proposition 2.10 from \cite{bcd} provides us with smooth radial functions $\chi$ and $\phii$, valued in the interval $[0,1]$, supported respectively\footnote{For $0<r<R$, we denote the ball $B(0,R) := \{\xi\in\R^d : |\xi|\le R\}$ and the annulus $\cont(r,R) := \{\xi\in\R^d : r\le|\xi|\le R\}$.} in $B:=B(0,4/3)$ and $\cont:=\cont(3/4,8/3)$, such that 
\[
\chi(\xi) + \sum_{j\ge0}\phii(2^{-j}\xi) = 1,\qquad\forall\,\xi\in\R^d.
\]
We then define the dyadic blocks\footnote{We denote by $f(D)$ the pseudo-differential operator defined, for any $u\in\sch'(\R^d)$ and all $\xi\in\R^d$, by $\widehat{f(D)u}(\xi) = f(\xi)\hat{u}(\xi)$.} $(\del_j)_{j\in\Z}$ by 
\[
\del_j := 0 \qquad\forall\,j\le-2,\qquad \del_{-1} := \chi(D),\qquad \del_j := \phii(2^{-j}D) \qquad\forall\,j\ge0,
\]
so that, for any $u\in\sch'(\R^d)$, we have the so-called \textit{Littlewood-Paley decomposition}
\[
u = \sum_{j\in\Z}\del_ju.
\]
We also define the low frequency cut-off operators $(S_j)_{j\in\Z}$ by 
\[
S_j := 0 \qquad\forall\,j\le-1,\qquad S_j := \chi(2^{-j}D) = \sum_{k\le j-1}\del_k \qquad\forall\,j\ge0.
\]
We now state some localisation properties for these operators. Since  
\[
\supp\phii(2^{-j}\cdot) \subset 2^j\cont\qquad\mathrm{and}\qquad\supp\chi(2^{-j}\cdot) \subset 2^jB,
\]
we have, for any $u,v\in\sch'(\R^d)$,
\begin{equation}
    \label{supp}
    \supp\widehat{\del_ju} \subset 2^j\cont,\qquad\supp\widehat{S_ju} \subset 2^jB,\qquad \supp\widehat{S_{j-1}u\del_jv} \subset 2^j\widetilde\cont,
\end{equation}
where we have defined $\widetilde\cont:=\cont(1/12,10/3)$. Since 
\[
2^j\cont\cap 2^k\cont=\emptyset\quad \forall\,|j-k|\ge2\qquad\mathrm{and}\qquad2^j\cont\cap 2^k\widetilde\cont=\emptyset\quad \forall\,|j-k|\ge5,
\]
we deduce that 
\begin{equation}
    \label{deltloc}
    \del_j\del_ku=0\quad\forall\,|j-k|\ge2,\qquad\del_j(S_{k-1}u\del_kv)=0\quad\forall\,|j-k|\ge5.
\end{equation}
As the operators $\del_j$ and $S_j$ are convolution operators, they map continuously $L^p(\R^d)$ into itself, for all $1\le p\le\infty$. Moreover, their norms are independent of $j$: there exists a constant $C>0$ such that for all $j\in\Z$,  
\begin{equation}
    \label{deltcont}
    \|\del_ju\|_{L^p} \le C\|u\|_{L^p}\qquad\mathrm{and}\qquad\|S_ju\|_{L^p} \le C\|u\|_{L^p}.
\end{equation}

We finally introduce the so-called \textit{Bernstein inequalities}.

\begin{lemma} {\bf (Bernstein inequalities).}
    Let $0<r<R$. There exists a constant $C>0$ such that, for any $k\ge0$, $1\le p\le q\le\infty$ and $u\in L^p(\R^d)$, we have, for all $\lambda>0$, 
    \[
    \begin{split}
        \supp\hat{u}\subset B(0,\lambda R)\qquad &\impl\qquad \|\nb^ku\|_{L^q} \le C^{k+1}\lambda^{k+d\left(\frac{1}{p}-\frac{1}{q}\right)}\|u\|_{L^p}, \\
        \supp\hat{u}\subset \cont(\lambda r,\lambda R)\qquad &\impl\qquad C^{-(k+1)}\lambda^k\|u\|_{L^p} \le \|\nb^ku\|_{L^p} \le C^{k+1}\lambda^k\|u\|_{L^p}.
    \end{split}
    \]
\end{lemma}
In particular, in view of \eqref{supp}, for any $1\le p\le q\le\infty$ and any $u\in L^p$, we have
\begin{equation}
    \label{Bernstein}
    \|\nb^k\del_{-1}u\|_{L^q}\lesssim \|\del_{-1}u\|_{L^p}\qquad\mathrm{and}\qquad\|\nb^k\del_ju\|_{L^p} \approx 2^{jk}\|\del_ju\|_{L^p},\qquad\forall\,j,k\ge0.
\end{equation}

\paragraph{Besov spaces.} Let $s\in\R$ and $1\le p,r\le\infty$. The Besov space $B_{p,r}^s(\R^d)$ is defined as the subset of tempered distributions $u\in\sch'(\R^d)$ such that 
\[
\|u\|_{B_{p,r}^s} := \left\|\left(2^{js}\|\del_ju\|_{L^p}\right)_{j\in\Z}\right\|_{\ell^r} < \infty.
\]

As a fundamental consequence of the Bernstein inequalities, we have the following embeddings between Besov spaces.
\begin{lemma}
\label{embBesov}
    Let $s_1,s_2\in\R$ and $1\le p_1,r_1,p_2,r_2\le\infty$. We have the continuous embedding $B_{p_1,r_1}^{s_1}\into B_{p_2,r_2}^{s_2}$ for all indices satisfying $p_1\le p_2$ and 
    \[
    s_1-\frac{d}{p_1}>s_2-\frac{d}{p_2},\qquad \text{or}\qquad s_1-\frac{d}{p_1}=s_2-\frac{d}{p_2}\quad\text{and}\quad r_1\le r_2.
    \]
\end{lemma}

Let us now recall some relations between Besov and Sobolev spaces. For all $k\ge0$ and $1\le p\le\infty$, we have the chain of continuous embeddings 
\[
B_{p,1}^k\into W^{k,p}\into B_{p,\infty}^k,\qquad \forall\,k\ge0,\qquad \forall\,1\le p\le\infty.
\]
In particular, $B_{\infty,1}^0\into\li\into B_{\infty,\infty}^0$. When $1<p<\infty$, we have the refined embeddings
\[
B_{p,\min(p,2)}^k\into W^{k,p}\into B_{p,\max(p,2)}^k,\qquad \forall\,k\ge0,\qquad \forall\,1<p<\infty,
\]
and for all $s\in\R$, we have $\hs \equiv B_{2,2}^s$, with equivalence of norms: for all $u\in\hs$, 
\[
\|u\|_\hs \approx \Big(\sum_{j\in\Z}2^{2js}\|\del_ju\|_\ld^2\Big)^{1/2}.
\]

As an important consequence of these relations, we have 
\[
\hs\equiv B_{2,2}^s\into B_{\infty,1}^0\into\li,\qquad\forall\,s>\frac{d}{2}.
\]

\subsection{Time-space Besov spaces and paradifferential calculus}
\label{sec:timespace}

Let $T>0$. Let $s\in\R$ and $1\le p,q,r\le\infty$. We introduce the time-space Besov space $\widetilde{L_T^q}B_{p,r}^s(\R^d)$ (after J.-Y. Chemin and N. Lerner, see \cite{cheminlerner}), defined as the subset of tempered distributions $u\in\sch'([0,T]\times\R^d)$ such that 
\[
\|u\|_{\widetilde{L_T^q}B_{p,r}^s} := \left\|\left(2^{js}\|\del_ju\|_{L_T^qL^p}\right)_{j\in\Z}\right\|_{\ell^r} < \infty.
\]
We also denote $\ctt B_{p,r}^s := \litt B_{p,r}^s\cap\ct B_{p,r}^s$.

\medskip
The use of these spaces is crucial for the proof of Theorem \ref{thstokes}. It comes from the fact that, when establishing a priori estimates for evolution equations, one has to apply first the frequency localisation operators $\del_j$, and perform an energy estimate. This provides one with $L^p$-estimates for each dyadic block $\del_ju$. In particular, one has to estimate a term of the form $\ddt\|\del_ju\|_{L^p}$. Therefore, we have to integrate in time \emph{before} performing the $\ell^r$-summation. This is the reason why the time integration and the $\ell^r$-summation are swapped in the definition of $\widetilde{L_T^q}B_{p,r}^s$, when compared to the usual spaces $L_T^qB_{p,r}^s$.

\medskip
We now make the link with the standard Besov spaces. From the Minkowski inequality, we have the continuous embedding $L_T^qB_{p,r}^s\into\widetilde{L_T^q}B_{p,r}^s$ if $q\le r$, and $\widetilde{L_T^q}B_{p,r}^s\into L_T^qB_{p,r}^s$ if $q\ge r$, with
\begin{equation}
    \label{timespace}
    \|u\|_{\widetilde{L_T^q}B_{p,r}^s} \le \|u\|_{L_T^qB_{p,r}^s}\qquad\mathrm{if}\;q\le r,\qquad\mathrm{and}\qquad \|u\|_{L_T^qB_{p,r}^s} \le \|u\|_{\widetilde{L_T^q}B_{p,r}^s}\qquad\mathrm{if}\;q\ge r.
\end{equation} 
As an immediate consequence, we have $\widetilde{L_T^q}B_{p,r}^s=L_T^qB_{p,r}^s$ whenever $q=r$, and in particular $\widetilde{L_T^2}\hs=L_T^2\hs$. We also have the continuous embeddings $\lut\hs\into\lutt\hs$ and $\litt\hs\into\lit\hs$, with
\begin{equation}
    \label{timespace1}
    \|u\|_{\lutt\hs} \le \|u\|_{\lut\hs}\qquad\mathrm{and}\qquad \|u\|_{\lit\hs} \le \|u\|_{\litt\hs},
\end{equation}
From Lemma \ref{embBesov} and \eqref{timespace}, we also have, for all $\delta>0$, the chain of continuous embeddings $\lutt H^{s+\delta}\into\lutt B_{2,1}^s=\lut B_{2,1}^s\into\lut\hs$, and 
\begin{equation}
    \label{timespace2}
    \|u\|_{\lut\hs} \lesssim \|u\|_{\lutt H^{s+\delta}}.
\end{equation}

Let us now state the interpolation inequality for time-space Besov spaces: for $s_1,s_2\in\R$, $1\le q_1,q_2\le\infty$ and $\theta\in[0,1]$, we have 
\begin{equation}
    \label{interpolation}
    \|u\|_{\widetilde{L_T^q}H^s} \le \|u\|_{\widetilde{L_T^{q_1}}H^{s_1}}^\theta \|u\|_{\widetilde{L_T^{q_2}}H^{s_2}}^{1-\theta}\qquad\mathrm{whenever}\qquad \frac{\theta}{q_1}+\frac{1-\theta}{q_2}=\frac{1}{q}\quad\mathrm{and}\quad\theta s_1+(1-\theta)s_2=s.
\end{equation}

We now present the so-called Fatou property of time-space Besov spaces. This result can be obtained with slight modifications in the proof of \cite[Theorem 2.25]{bcd}.

\begin{theorem}
    \label{th:fatou}
    Let $T>0$. Let $s\in\R$ and $1\le p,q,r\le\infty$. Let $(u_n)_{n\ge0}$ be a bounded sequence in $\widetilde{L_T^q}B_{p,r}^s$. There exists $u\in\widetilde{L_T^q}B_{p,r}^s$ and an extraction $\phii$ such that 
    \[
    u_{\phii(n)}\to u\qquad\mathrm{in}\;\sch'([0,T]\times\R^d)\qquad\text{and}\qquad\|u\|_{\widetilde{L_T^q}B_{p,r}^s}\lesssim\liminf_{n\to+\infty}\|u_{\phii(n)}\|_{\widetilde{L_T^q}B_{p,r}^s}.
    \]
\end{theorem}

\paragraph{Elements of paradifferential calculus.} We introduce the paraproduct decomposition (after J.-M. Bony, see \cite{bony}) in the framework of time-space Besov spaces. The product of two tempered distributions $u,v\in\sch'([0,T]\times\R^d)$ can be decomposed as 
\begin{equation}
    \label{paraprod}
    uv = \T_uv + \T_vu + \RR(u,v),
\end{equation}
where 
\[
\T_uv(t) := \sum_{j\in\Z} S_{j-1}u(t)\,\del_jv(t)\qquad\mathrm{and}\qquad \RR(u,v)(t) := \sum_{j\in\Z} \sum_{|k-j|\le1} \del_ju(t)\,\del_kv(t)
\]
are called respectively the \textit{paraproduct} and \textit{remainder} of $u$ and $v$.

\medskip
Let us now recall some well-known continuity properties for these operators. The following result is adapted from \cite[Theorems 2.82 and 2.85]{bcd}.
\begin{proposition}
    \label{para2}
    Let $T>0$. Let $s\in\R$ and $1\le p,q,r\le\infty$. Let $1\le q_1,q_2\le\infty$ be such that 
    \[
    \frac{1}{q_1} + \frac{1}{q_2} = \frac{1}{q}.
    \]
    For any $\ell>0$, the paraproduct operator $\T$ maps continuously $L_T^{q_1}\li\times\widetilde{L_T^{q_2}}B_{p,r}^s$ and $L_T^{q_1}B_{\infty,\infty}^{-\ell}\times\widetilde{L_T^{q_2}}B_{p,r}^{s+\ell}$ into $\widetilde{L_T^q}B_{p,r}^s$, and we have, for all $k\ge0$,
    \[
    \|\T_uv\|_{\widetilde{L_T^q}B_{p,r}^s} \lesssim \|u\|_{L_T^{q_1}\li} \|\nb^kv\|_{\widetilde{L_T^{q_2}}B_{p,r}^{s-k}}\qquad\text{and}\qquad \|\T_uv\|_{\widetilde{L_T^q}B_{p,r}^s} \lesssim \|u\|_{L_T^{q_1}B_{\infty,\infty}^{-\ell}} \|\nb^kv\|_{\widetilde{L_T^{q_2}}B_{p,r}^{s-k+\ell}}.
    \]
    Let $s_1,s_2\in\R$ and $1\le p_1,p_2,r_1,r_2\le\infty$ be such that
    \[
    s_1+s_2=s,\qquad \frac{1}{p_1} + \frac{1}{p_2} = \frac{1}{p},\qquad \frac{1}{r_1} + \frac{1}{r_2} = \frac{1}{r}.
    \]
    If $s>0$, the remainder operator $\RR$ maps continuously $\widetilde{L_T^{q_1}}B_{p_1,r_1}^{s_1}\times\widetilde{L_T^{q_2}}B_{p_2,r_2}^{s_2}$ into $\widetilde{L_T^q}B_{p,r}^s$, and we have 
    \[
    \|\RR(u,v)\|_{\widetilde{L_T^q}B_{p,r}^s} \lesssim \|u\|_{\widetilde{L_T^{q_1}}B_{p_1,r_1}^{s_1}} \|v\|_{\widetilde{L_T^{q_2}}B_{p_2,r_2}^{s_2}}.
    \]
    If $s=0$ and $r=1$, the operator $\RR$ maps continuously $\widetilde{L_T^{q_1}}B_{p_1,r_1}^{s_1}\times\widetilde{L_T^{q_2}}B_{p_2,r_2}^{s_2}$ into $\widetilde{L_T^q}B_{p,\infty}^0$, and we have
    \[
    \|\RR(u,v)\|_{\widetilde{L_T^q}B_{p,\infty}^0} \lesssim \|u\|_{\widetilde{L_T^{q_1}}B_{p_1,r_1}^{s_1}} \|v\|_{\widetilde{L_T^{q_2}}B_{p_2,r_2}^{s_2}}.
    \]
\end{proposition}

As an immediate and fundamental consequence of Proposition \ref{para2}, we have the so-called \textit{tame estimates}. 

\begin{corollary}
    \label{tameestimates}
    Let $T>0$. Let $s>0$ and $1\le p,q,r\le\infty$. Let $1\le q_1,q_2,q_3,q_4\le\infty$ be such that 
    \[
    \frac{1}{q_1} + \frac{1}{q_2} = \frac{1}{q_3} + \frac{1}{q_4} = \frac{1}{q}.
    \] 
    We have, for all $k\ge0$,
    \[
    \|uv\|_{\widetilde{L_T^q}B_{p,r}^s} \lesssim \|u\|_{L_T^{q_1}\li} \|v\|_{\widetilde{L_T^{q_2}}B_{p,r}^s} + \|\nb^ku\|_{\widetilde{L_T^{q_3}}B_{p,r}^{s-k}} \|v\|_{L_T^{q_4}\li}.
    \]
\end{corollary}

In particular, with the choices $p=r=2$, $q=q_2=q_4=1$, $q_1=q_3=\infty$, and $k=1$, we have 
\begin{equation}
    \label{eq:tame}
    \|uv\|_{\lutt\hs} \lesssim \|u\|_{\lit\li} \|v\|_{\lutt\hs} + \|\nb u\|_{\litt\hsm} \|v\|_{\lut\li}.
\end{equation}
If $s>d/2$, owing to embeddings \eqref{timespace1}-\eqref{timespace2}, we have, for any $\delta>0$ such that $s>d/2+\delta$, 
\[
\|u\|_{\lit\li} \lesssim \|u\|_{\lit\hs} \le \|u\|_{\litt\hs}\qquad\mathrm{and}\qquad \|v\|_{\lut\li} \lesssim \|v\|_{\lut H^{d/2+\delta}} \lesssim \|v\|_{\lutt\hs},
\]
so that 
\begin{equation}
    \label{eq:tame2}
    \|uv\|_{\lutt\hs} \lesssim \|u\|_{\litt\hs} \|v\|_{\lutt\hs}.
\end{equation}

We also recall the following commutator estimates, which are particular cases of \cite[Lemmas 8.7 and 8.11]{danchin2006}.

\begin{proposition}
\label{commut}
    Let $T>0$. Let $s>1+d/2$. Assume that $\divv\,v=0$. There exists a constant $C=C(d,s)>0$ such that 
    \[
    \Big\|\Big(2^{js}\big\|[\del_j,a]w\big\|_{\lut\ld}\Big)_{j\ge-1}\Big\|_{\ell^2} \le C\|\nb a\|_{\litt\hs}\|w\|_{\lutt H^{s-1/2}},
    \]
    and
    \[
    \Big\|\Big(2^{js}\big\|[v\cdot\nb,\del_j]w\big\|_{\lut\ld}\Big)_{j\ge-1}\Big\|_{\ell^2} \le C\int_0^T \|\nb v(t)\|_\hsm \|\nb w(t)\|_\hsm \dd t.
    \]
\end{proposition}

\paragraph{A paralinearisation estimate.} We now present the following result on composition of functions in time-space Besov spaces, adapted from \cite[Proposition 4]{danchin2010}. 

\begin{proposition}
    \label{compo2}
    Let $T>0$. Let $s>0$ and $1\le p,q,r\le\infty$. Let $F\in C^{\ps+2}(\R)$. For all $a\in\li([0,T]\times\R^d)$ with $\nb a\in\widetilde{L_T^q}B_{p,r}^{s-1}(\R^d)$, we have $\nb F(a)\in\widetilde{L_T^q}B_{p,r}^{s-1}(\R^d)$, and there exists a constant $C=C(s,F',\|a\|_{\lit\li})>0$, depending only on the quantities inside the brackets, such that 
    \[
    \|\nb F(a)\|_{\widetilde{L_T^q}B_{p,r}^{s-1}} \le C\|\nb a\|_{\widetilde{L_T^q}B_{p,r}^{s-1}}.
    \]
\end{proposition}

As this proposition dictates the level of regularity one has to assume on the function $f$ in Theorem \ref{main}, we want to be precise regarding the required regularity on $F$. In the statement of \cite{danchin2010}, the function $F$ is supposed to be smooth, which is more than what is actually needed. In order to justify our refined version of this statement, we provide the full proof, which however follows the same lines as those of \cite[Proposition 4]{danchin2010} and \cite[Theorem 2.61]{bcd}.

\begin{proof}
    In view of \eqref{deltcont}, we fix a constant $C\ge1$ such that for all $j\in\Z$, we have  
    \begin{equation}
        \label{deltcontcomp}
        \|\del_ja\|_{\lit\li} \le C\|a\|_{\lit\li}\qquad\mathrm{and}\qquad \|S_ja\|_{\lit\li} \le C\|a\|_{\lit\li}.
    \end{equation}
    {\bf Step 1. Decomposition of $F(a)$.} Let $n\ge1$. We have 
    \[
    \sum_{j=1}^n \big(F(S_{j+1}a)-F(S_ja)\big) = F(S_{n+1}a)-F(S_1a).
    \]
    Now, since $F'$ is bounded on $B(0,2C\|a\|_{\lit\li})$, we have by the mean value theorem that 
    \[
    \|F(a)-F(S_{n+1}a)\|_{L_T^qL^p} \lesssim \|a-S_{n+1}a\|_{L_T^qL^p}.
    \]
    We can then use the second Bernstein inequality in \eqref{Bernstein} and the Hölder inequality (as $s>0$) to write that
    \[
    \|a-S_{n+1}a\|_{L_T^qL^p} \le \sum_{j\ge n+1}\|\del_ja\|_{L_T^qL^p} \lesssim \sum_{j\ge n+1}2^{-j}\|\del_j\nb a\|_{L_T^qL^p} \lesssim \Big(\sum_{j\ge n+1}2^{j(s-1)r}\|\del_j\nb a\|_{L_T^qL^p}^r\Big)^{1/r},
    \]
    which vanishes to $0$ as $n\to+\infty$, as $\nb a\in {\widetilde{L_T^q}B_{p,r}^{s-1}}$. Thus, $F(S_{n+1}a)$ converges to $F(a)$ in ${L_T^qL^p}$, and we have 
    \[
    \sum_{j\ge1} \big(F(S_{j+1}a)-F(S_ja)\big) = F(a)-F(S_1a).
    \]
    This can be rewritten as 
    \[
    F(a) = F(S_1a) + \sum_{j\ge1}m_j\del_ja,\qquad\mathrm{with}\qquad m_j := \int_0^1 F'(S_ja+t\del_ja)\dd t.
    \]
    From the localisation property \eqref{deltloc}, we then have 
    \[
    F(a)-F(S_1a) = \sum_{j\ge1}F_j,\qquad\mathrm{with}\qquad F_j:=m_j\del_j(a-\del_{-1}a).
    \]
    {\bf Step 2. Estimate of $\|\nb F(a)-\nb F(S_1a)\|_{\widetilde{L_T^q}B_{p,r}^{s-1}}$.} For any fixed $j\in\Z$, we can thus write
    \[
    2^{js}\del_j\big(F(a)-F(S_1a)\big) = 2^{js}\sum_{1\le k\le j}\del_jF_k + 2^{js}\sum_{k>j}\del_jF_k.
    \]
    On the one hand, using \eqref{deltcont}, we write that 
    \[
    2^{js}\sum_{k>j}\|\del_jF_k\|_{L_T^qL^p} \lesssim \sum_{k>j} 2^{(j-k)s}2^{ks}\|F_k\|_{L_T^qL^p} = \sum_{k\in\Z}b_{j-k}c_k = (b\ast c)_j,
    \]
    with the convention $F_j:=0$ for all $j\le0$, where the sequences $b=(b_n)_{n\in\Z}$ and $c=(c_n)_{n\in\Z}$ are defined by 
    \[
    b_n := \mathds{1}_{(-\infty,0[}(n)2^{ns},\qquad c_n := 2^{ns}\|F_n\|_{L_T^qL^p}.
    \]
    On the other hand, by the second Bernstein inequality in \eqref{Bernstein}, we have 
    \[
    2^{js}\sum_{1\le k\le j}\|\del_jF_k\|_{L_T^qL^p} \lesssim \sum_{1\le k\le j} 2^{-(j-k)(\ps+1-s)} \sup_{|\alpha|=\ps+1}2^{k(s-|\alpha|)}\|\pt^\alpha F_k\|_{L_T^qL^p} = \sum_{k\in\Z} d_{j-k}e_k = (d\ast e)_j,
    \]
    where the sequences $d=(d_n)_{n\in\Z}$ and $e=(e_n)_{n\in\Z}$ are defined by 
    \[
    d_n := \mathds{1}_{[0,\infty)}(n)2^{-n(\ps+1-s)},\qquad e_n := \sup_{|\alpha|=\ps+1}2^{n(s-|\alpha|)}\|\pt^\alpha F_n\|_{L_T^qL^p}.
    \]
    We then have, for all $j\in\Z$, 
    \[
    2^{js}\left\|\del_j\big(F(a)-F(S_1a)\big)\right\|_{L_T^qL^p} \lesssim (b\ast c)_j + (d\ast e)_j,
    \]
    from which we gather, after taking the $\ell^r$-norm and using the Young inequality, that 
    \begin{equation}
        \label{fafsua}
        \|F(a)-F(S_1a)\|_{\widetilde{L_T^q}B_{p,r}^s} \lesssim \Big\|\Big(\sup_{|\alpha|\in\{0,\ps+1\}}2^{j(s-|\alpha|)}\|\pt^\alpha F_j\|_{L_T^qL^p}\Big)_{j\in\Z}\Big\|_{\ell^r}.
    \end{equation}
    Let $\alpha\in\N^d$, $|\alpha|\le\ps+1$. From the Leibniz formula, we have 
    \[ 
    \pt^\al F_j = \pt^\al\left(m_j\del_j\big(a-\del_{-1}a\big)\right) = \sum_{\beta\le\alpha} \binom{\alpha}{\beta}\,\pt^\beta m_j\,\pt^{\alpha-\beta}\del_j\big(a-\del_{-1}a\big).
    \]
    Using the Hölder and Bernstein inequalities, we then obtain 
    \begin{equation}
        \label{leibnizdalfj}
        \|\pt^\al F_j\|_{L_T^qL^p} \lesssim \sum_{\beta\le\alpha} \|\pt^\beta m_j\|_{\lit\li}\,2^{j(|\alpha|-|\beta|)}\left\|\del_j\big(a-\del_{-1}a\big)\right\|_{L_T^qL^p}.
    \end{equation}
    Now, define $G:\R^2\to\R$ by 
    \[
    G(x,y) = \int_0^1 F'(x+ty)\dd t,
    \]
    and $\Theta:=(S_ja,\del_ja):\R^d\to\R^2$, so that $m_j=G(\Theta)$. From the Faà di Bruno formula \cite[Lemma 2.3]{bcd}, we have 
    \[
    \pt^\beta m_j = \pt^\beta G(\Theta) = \sum_{\mu,\nu}C_{\mu,\nu}(\pt^\mu G)(\Theta) \prod_{\substack{1\le|\gamma|\le|\beta| \\ k=1,2}}(\pt^\gamma\Theta_k)^{\nu_{\gamma_k}},
    \]
    where $C_{\mu,\nu}\in\N$, and the sum is taken over those $\mu\in\N^2$ such that $1\le|\mu|\le|\beta|$, and those $\nu=(\nu_{\gamma_k})_{\substack{1\le|\gamma|\le|\beta| \\ k=1,2}} $ with $\nu_{\gamma_k}\in\N^*$ satisfying
    \begin{equation}
        \label{identfaa}
        \sum_{1\le|\gamma|\le|\beta|}\nu_{\gamma_k} = \mu_k\quad(k=1,2),\qquad\mathrm{and}\qquad \sum_{\substack{1\le|\gamma|\le|\beta| \\ k=1,2}}\gamma\nu_{\gamma_k} = \beta.
    \end{equation}
    By the Bernstein inequalities \eqref{Bernstein}, we then have  
    \[
    \|\pt^\beta m_j\|_{\lit\li} \lesssim \sum_{\mu,\nu}\|(\pt^\mu G)(\Theta)\|_{\lit\li} \prod_{\substack{1\le|\gamma|\le|\beta| \\ k=1,2}}2^{j|\gamma|\nu_{\gamma_k}}\|\Theta_k\|_{\lit\li}^{\nu_{\gamma_k}}.
    \]
    Now, from \eqref{deltcontcomp}, we have 
    \[
    \Theta\in B(0,\sqrt2C\|a\|_{\lit\li}),\qquad\mathrm{and}\qquad \|\Theta_k\|_{\lit\li}\le C\|a\|_{\lit\li}\quad(k=1,2).
    \]
    Since all the derivatives up to order $\ps+1$ of $F'$, hence of $G$, are bounded on $B(0,\sqrt2C\|a\|_{\lit\li})$, we deduce that 
    \[
    \|\pt^\beta m_j\|_{\lit\li} \lesssim \sum_{\mu,\nu}\prod_{\substack{1\le|\gamma|\le|\beta| \\ k=1,2}}2^{j|\gamma|\nu_{\gamma_k}} \lesssim 2^{j|\beta|},
    \]
    where we have also used the second identity in \eqref{identfaa} for the last inequality. Plugging this into \eqref{leibnizdalfj}, we gather that 
    \[
    \|\pt^\al F_j\|_{L_T^qL^p} \lesssim 2^{j|\alpha|}\left\|\del_j\big(a-\del_{-1}a\big)\right\|_{L_T^qL^p}.
    \]
    Plugging this in turn into \eqref{fafsua} now yields 
    \[
    \|F(a)-F(S_1a)\|_{\widetilde{L_T^q}B_{p,r}^s} \lesssim \|a-\del_{-1}a\|_{\widetilde{L_T^q}B_{p,r}^{s}}.
    \]
    From the localisation property \eqref{deltloc}, we have, for any $j\in\Z$, 
    \[
    \del_j\big(a-\del_{-1}a\big) = \sum_{\substack{k\ge0 \\ |k-j|\le1}}\del_j\del_ka. 
    \]
    By the second Bernstein inequality in \eqref{Bernstein}, we deduce that 
    \[
    \left\|\del_j\big(a-\del_{-1}a\big)\right\|_{L_T^qL^p} \lesssim 2^{-j}\|\del_j\nb a\|_{L_T^qL^p},\qquad\mathrm{thus}\qquad \|a-\del_{-1}a\|_{\widetilde{L_T^q}B_{p,r}^{s}} \lesssim \|\nb a\|_{\widetilde{L_T^q}B_{p,r}^{s-1}}.
    \]
    All in all, we have 
    \begin{equation}
        \label{fafsua2}
        \|\nb F(a)-\nb F(S_1a)\|_{\widetilde{L_T^q}B_{p,r}^{s-1}} \lesssim \|F(a)-F(S_1a)\|_{\widetilde{L_T^q}B_{p,r}^s} \lesssim \|\nb a\|_{\widetilde{L_T^q}B_{p,r}^{s-1}}.
    \end{equation}
    {\bf Step 3. Estimate of $\|\nb F(S_1a)\|_{\widetilde{L_T^q}B_{p,r}^{s-1}}$.} Now, write that 
    \begin{equation}
        \label{nbfsa}
        \|\nb F(S_1a)\|_{\widetilde{L_T^q}B_{p,r}^{s-1}}^r = 2^{-(s-1)r}\|\del_{-1}\nb F(S_1a)\|_{L_T^qL^p}^r + \sum_{j\ge0}2^{j(s-1)r}\|\del_j\nb F(S_1a)\|_{L_T^qL^p}^r.
    \end{equation}
    As for the low frequency term, from \eqref{deltcont}, we have 
    \begin{equation}
        \label{lfnbfsa}
        \|\del_{-1}\nb F(S_1a)\|_{L_T^qL^p} \lesssim \|F'(S_1a)\nb S_1a\|_{L_T^qL^p} \lesssim \|\nb S_1a\|_{L_T^qL^p},
    \end{equation}
    where we have also used the fact that $F'$ is bounded on $B(0,C\|a\|_{\lit\li})$. As for the high frequencies, from the second Bernstein inequality in \eqref{Bernstein} and \eqref{deltcont}, we have, for any $j\ge0$, 
    \[
    2^{j(s-1)}\|\del_j\nb F(S_1a)\|_{L_T^qL^p} \lesssim 2^{-j(\ps-s+1)}\max_{|\alpha|=\ps}\|\pt^\alpha(F'(S_1a)\nb S_1a)\|_{L_T^qL^p}.
    \]
    From the Leibniz formula, we have 
    \[
    \|\pt^\alpha(F'(S_1a)\nb S_1a)\|_{L_T^qL^p} \lesssim \sum_{\beta\le\alpha}\|\pt^\beta F'(S_1a)\|_{\lit\li}\|\pt^{\alpha-\beta}\nb S_1a\|_{L_T^qL^p}.
    \]
    From the Faà di Bruno formula, it holds that
    \[
    \|\pt^\beta F'(S_1a)\|_{\lit\li} \lesssim \sum_{\mu,\nu} \|F^{(1+\mu)}(S_1a)\|_{\lit\li} \prod_{1\le|\gamma|\le|\beta|}\|\pt^\gamma S_1a\|_{\lit\li}^{\nu_\gamma},
    \]
    where the sum is taken over those $\mu\in\N$ such that $1\le\mu\le|\beta|$, and those $\nu=(\nu_\gamma)_{1\le|\gamma|\le|\beta|}$ with $\nu_\gamma\in\N^*$ satisfying
    \[
    \sum_{1\le|\gamma|\le|\beta|}\nu_\gamma = \mu \qquad\mathrm{and}\qquad \sum_{1\le|\gamma|\le|\beta|}\gamma\nu_\gamma = \beta.
    \]
    Since all the derivatives up to order $\ps$ of $F'$ are bounded on $B(0,C\|a\|_{\lit\li})$, we gather that 
    \[
    \|\pt^\beta F'(S_1a)\|_{\lit\li} \lesssim 1.
    \]
    From the first Bernstein inequality in \eqref{Bernstein}, we then obtain
    \begin{equation}
        \label{hfnbfsa}
        2^{j(s-1)}\|\del_j\nb F(S_1a)\|_{L_T^qL^p} \lesssim 2^{-j(\ps-s+1)}\|\nb S_1a\|_{L_T^qL^p}.
    \end{equation}
    Since $\ps-s+1>0$, plugging \eqref{lfnbfsa} and \eqref{hfnbfsa} into \eqref{nbfsa} yields 
    \[
    \|\nb F(S_1a)\|_{\widetilde{L_T^q}B_{p,r}^{s-1}} \lesssim \|\nb S_1a\|_{L_T^qL^p}.
    \]
    Now, from the localisation property \eqref{deltloc}, we have 
    \[
    \nb S_1a = \sum_{-1\le j\le1} \del_j\nb S_1a,
    \]
    so that, using the second inequality in \eqref{deltcont}, we obtain
    \[
    \|\nb S_1a\|_{L_T^qL^p}^r \lesssim \sum_{-1\le j\le1} 2^{j(s-1)r}\|\del_j\nb a\|_{L_T^qL^p}^r \le \|\nb a\|_{\widetilde{L_T^q}B_{p,r}^{s-1}}^r.
    \]
    We thus have 
    \[
    \|\nb F(S_1a)\|_{\widetilde{L_T^q}B_{p,r}^{s-1}} \lesssim \|\nb a\|_{\widetilde{L_T^q}B_{p,r}^{s-1}}.
    \]
    This completes the proof.  
\end{proof}

\subsection{Elliptic and transport estimates}

To complete this appendix, we present some elliptic and transport estimates in Sobolev spaces $\hs(\R^d)$, which are particular cases of more general results in Besov spaces $B_{p,r}^s(\R^d)$. 

\paragraph{Elliptic estimates.} We consider the elliptic equation 
\begin{equation}
    \label{elliptic}
    -\divv(a\nb\Pi)=\divv(F)\qquad\mathrm{in}\;\R^d,
\end{equation}
where $a=a(x)$ is a given smooth bounded function satisfying 
\begin{equation}
    \label{ellcond}
    a_* := \inf_{x\in\R^d}a(x)>0. 
\end{equation}

We have the following result, see \cite[Lemma 2]{danchin2010}.

\begin{lemma}
    \label{ellipticlow}
    For all vector fields $F\in\ld(\R^d)$, there exists a unique (up to constant functions) tempered distribution $\Pi$ satisfying equation \eqref{elliptic}, with $\nb\Pi\in\ld(\R^d)$, satisfying the estimate 
    \[
    a_*\|\nb\Pi\|_\ld\le\|F\|_\ld.
    \]
\end{lemma}

We now state some higher order estimates for equation \eqref{elliptic}, see \cite[Proposition 7]{danchin2010}.

\begin{proposition}
    \label{elliptichigh}
    Let $s>1+d/2$. Let $a$ be a bounded function satisfying \eqref{ellcond}, and such that $\nb a\in\hsm(\R^d)$. Let $F\in\ld(\R^d)$ be such that $\divv(F)\in\hsm(\R^d)$. Then, equation \eqref{elliptic} admits a unique (up to constant functions) solution $\Pi$ such that $\nb\Pi\in\hs(\R^d)$, and there exist constants $C=C(d,s)>0$ and $\gamma=\gamma(d,s)>0$ such that
    \[
    a_*\|\nb\Pi\|_\hs \le C\left(\|\divv(F)\|_\hsm + \left(1+a_*^{-1}\|\nb a\|_\hsm\right)^\gamma\|F\|_\ld\right).
    \]
\end{proposition}

\paragraph{Transport estimates.}

We consider the transport equation

\begin{equation}
    \label{transport}
    \left\{
    \begin{aligned}
    & (\dt + v\cdot\nb)f = g, \\
    & f_{|t=0}=f_0.
\end{aligned}
\right.
\end{equation}
\par
The following statement is a particular case of \cite[Theorem 3.14 and Remark 3.15]{bcd}.
\begin{theorem}
    \label{th:transport1}
    Let $T>0$ and $s>1+d/2$. Let $f_0\in\hs(\R^d)$ and $g\in\lutt\hs(\R^d)$. Let $v$ be a divergence-free vector field such that $\nb v\in\lut\hsm(\R^d)$. There exists a constant $C=C(d,s)>0$ such that any solution $f$ of equation \eqref{transport} satisfies
    \[
    \|f\|_{\litt\hs} \le Ce^{C\|\nb v\|_{\lut\hsm}}\Big(\|f_0\|_\hs + \|g\|_{\lutt\hs}\Big).
    \]
\end{theorem}

We now provide the following adaptation of \cite[Theorem 3.19]{bcd}.

\begin{theorem}
    \label{th:transport2}
    Let $T>0$ and $s>1+d/2$. Let $f_0\in\hs(\R^d)$ and $g\in\lut\hs(\R^d)$. Let $v$ be a divergence-free vector field such that $v\in L_T^qB_{\infty,\infty}^{-M}(\R^d)$, for some $q>1$ and $M>0$, with $\nb v\in\lut\hsm(\R^d)$. Then, equation \eqref{transport} has a unique solution $f\in\ct\hs(\R^d)$, and the inequality of Theorem \ref{th:transport1} is satisfied.
\end{theorem}

We finally state a result for estimating the gradient of a solution to \eqref{transport}, which can be obtained following the proof of \cite[Theorem 3.14]{bcd}. 
\begin{proposition}
    \label{trgrad}
    Let $T>0$ and $s>1+d/2$. Let $v$ be a divergence-free vector field such that $\nb v\in\lut\hs(\R^d)$. Let $f_0$ be such that $\nb f_0\in\hs(\R^d)$. Let $g$ be such that $\nb g\in\lutt\hs(\R^d)$. Then, there exists a constant $C=C(d,s)>0$ such that any solution $f$ of equation \eqref{transport} satisfies
    \[
    \|\nb f\|_{\litt\hs} \le Ce^{C\|\nb v\|_{\lut\hs}}\Big(\|\nb f_0\|_\hs + \|\nb g\|_{\lutt\hs}\Big).
    \]
\end{proposition}

\paragraph{Acknowledgements.} I am grateful to my advisor Francesco Fanelli for introducing me to this subject and for his technical advice. My visits at the Basque Center for Applied Mathematics (BCAM) in January and June 2025 were funded by BCAM and the project CRISIS (ANR-20-CE40-0020-01), operated by the French National Research Agency (ANR). I also thank the people at BCAM for their hospitality.



\addcontentsline{toc}{section}{References}

\begin{thebibliography}{10}

\bibitem{Abidi2007_Critique}
H.~Abidi, \emph{Équation de {Navier--Stokes} avec densité et viscosité variables dans l'espace critique}, Rev. Mat. Iberoam. \textbf{23} (2007), no.~2, 537--586.

\bibitem{AbidiGuiZhang2012}
H.~Abidi, G.~Gui, and P.~Zhang, \emph{On the well-posedness of three-dimensional inhomogeneous {Navier--Stokes} equations in the critical spaces}, Arch. Ration. Mech. Anal. \textbf{204} (2012), no.~1, 189--230.

\bibitem{AbidiPaicu2007}
H.~Abidi and M.~Paicu, \emph{Existence globale pour un fluide inhomog{\`e}ne}, Ann. Inst. Fourier \textbf{57} (2007), no.~3, 883--917.

\bibitem{avron1998odd}
J.~E. Avron, \emph{Odd viscosity}, J. Stat. Phys. \textbf{92} (1998), 543--557.

\bibitem{avron1995viscosity}
J.~E. Avron, R.~Seiler, and P.~G. Zograf, \emph{Viscosity of quantum {Hall} fluids}, Phys. Rev. Lett. \textbf{75} (1995), no.~4, 697.

\bibitem{bcd}
H.~Bahouri, J.-Y. Chemin, and R.~Danchin, \emph{{Fourier Analysis and Nonlinear Partial Differential Equations}}, Grundlehren der Mathematischen Wissenschaften, vol. 343, Springer, 2011.

\bibitem{chiralactive}
D.~Banerjee, A.~Souslov, A.~G. Abanov, and V.~Vitelli, \emph{Odd viscosity in chiral active fluids}, Nat. Commun. \textbf{8} (2017), no.~1, 1573.

\bibitem{batchelor1967introduction}
G.~K. Batchelor, \emph{{An Introduction to Fluid Dynamics}}, Cambridge University Press, 1967.

\bibitem{bony}
J.-M. Bony, \emph{Calcul symbolique et propagation des singularit\'{e}s pour les \'equations aux d\'{e}riv\'{e}es partielles non lin\'{e}aires}, Ann. Sci. Éc. Norm. Supér. \textbf{14} (1981), no.~2, 209--246.

\bibitem{bd2003}
D.~Bresch and B.~Desjardins, \emph{Existence of global weak solutions for a {2D} viscous shallow water equations and convergence to the quasi‐geostrophic model}, Comm. Math. Phys. \textbf{238} (2003), no.~1--2, 211--223.

\bibitem{bd2004}
D.~Bresch and B.~Desjardins, \emph{Quelques modèles diffusifs capillaires de type {Korteweg}}, C. R. Méc. \textbf{332} (2004), no.~11, 881--886.

\bibitem{bd2006}
D.~Bresch and B.~Desjardins, \emph{On the construction of approximate solutions for the {2D} viscous shallow water model and for compressible {Navier--Stokes} models}, J. Math. Pures Appl. \textbf{86} (2006), no.~4, 362--368.

\bibitem{bdl2003}
D.~Bresch, B.~Desjardins, and C.-K. Lin, \emph{On some compressible fluid models: {Korteweg}, lubrication, and shallow water systems}, Comm. Partial Differential Equations \textbf{28} (2003), no.~3-4, 843--868.

\bibitem{bvy2021}
D.~Bresch, A.~F. Vasseur, and C.~Yu, \emph{Global existence of entropy-weak solutions to the compressible {Navier--Stokes} equations with non-linear density dependent viscosities}, J. Eur. Math. Soc. \textbf{24} (2021), no.~5, 1791--1837.

\bibitem{cheminlerner}
J.-Y. Chemin and N.~Lerner, \emph{Flot de champs de vecteurs non lipschitziens et {\'e}quations de {Navier--Stokes}}, J. Differential Equations \textbf{121} (1995), no.~2, 314--328.

\bibitem{cobb2023elsasser}
D.~Cobb and F.~Fanelli, \emph{{Els{\"a}sser} formulation of the ideal {MHD} and improved lifespan in two space dimensions}, J. Math. Pures Appl. \textbf{169} (2023), 189--236.

\bibitem{danchin2003density}
R.~Danchin, \emph{Density-dependent incompressible viscous fluids in critical spaces}, Proc. Roy. Soc. Edinburgh Sect. A \textbf{133} (2003), no.~6, 1311--1334.

\bibitem{danchin2006}
R.~Danchin, \emph{The inviscid limit for density-dependent incompressible fluids}, Ann. Fac. Sci. Toulouse Math. (6) \textbf{15} (2006), no.~4, 637--688.

\bibitem{danchin2010}
R.~Danchin, \emph{On the well-posedness of the incompressible density-dependent {Euler} equations in the $\lp$ framework}, J. Differential Equations \textbf{248} (2010), no.~8, 2130--2170.

\bibitem{danchinlagrangian2012}
R.~Danchin, \emph{A {Lagrangian} approach for the compressible {Navier--Stokes} equations}, Ann. Inst. Fourier \textbf{64} (2014), 753--791.

\bibitem{df}
R.~Danchin and F.~Fanelli, \emph{The well-posedness issue for the density-dependent Euler equations in endpoint Besov spaces}, J. Math. Pures Appl. \textbf{96} (2011), no.~3, 253--278.

\bibitem{dm2019}
R.~Danchin and P.~B. Mucha, \emph{The incompressible Navier--Stokes equations in vacuum}, Comm. Pure Appl. Math. \textbf{72} (2019), no.~7, 1351--1385.

\bibitem{dm2023}
R.~Danchin and P.~B. Mucha, \emph{Compressible Navier--Stokes equations with ripped density}, Comm. Pure Appl. Math. \textbf{76} (2023), no.~11, 3437-3492.

\bibitem{Desjardins1997}
B.~Desjardins, \emph{Global existence results for the incompressible density-dependent {Navier--Stokes} equations}, Differential Integral Equations \textbf{10} (1997), no.~3, 587--598.

\bibitem{fanelli2012}
F.~Fanelli, \emph{Conservation of geometric structures for non-homogeneous inviscid incompressible fluids}, Comm. Partial Differential Equations \textbf{37} (2012), no.~9, 1553--1595.

\bibitem{fgs}
F.~Fanelli, R.~Granero-Belinchón, and S.~Scrobogna, \emph{Well-posedness theory for non-homogeneous incompressible fluids with odd viscosity}, J. Math. Pures Appl. \textbf{187} (2024), 58--137.

\bibitem{fv}
F.~Fanelli and A.~F. Vasseur, \emph{Effective velocity and $\li$-based well-posedness for incompressible fluids with odd viscosity}, SIAM J. Math. Anal. \textbf{57} (2025), no.~1, 153--189.

\bibitem{feireisl2004dynamics}
E.~Feireisl, \emph{{Dynamics of viscous compressible fluids}}, vol.~26, Oxford University Press, 2004.

\bibitem{fruchart2023odd}
M.~Fruchart, C.~Scheibner, and V.~Vitelli, \emph{Odd viscosity and odd elasticity}, Annu. Rev. Condens. Matter Phys. \textbf{14} (2023), no.~1, 471--510.

\bibitem{gancedo20252dnavierstokesfreeboundary}
F.~Gancedo, E.~García-Juárez, and P.~Luna-Velasco, \emph{On {2D Navier--Stokes} free boundary: nonnegative density and small viscosity contrast}, preprint, \href{https://arxiv.org/abs/2507.09333}{\texttt{arXiv:2507.09333}} (2025).

\bibitem{GraneroBelinchonOrtega2022}
R.~Granero-Belinchón and A.~Ortega, \emph{On the motion of gravity–capillary waves with odd viscosity}, J. Nonlinear Sci. \textbf{32} (2022), 28.

\bibitem{HuangWang2015}
X.~Huang and Y.~Wang, \emph{Global strong solution of {3D inhomogeneous Navier--Stokes} equations with density-dependent viscosity}, J. Differential Equations \textbf{259} (2015), no.~4, 1606--1627.

\bibitem{landau1987fluid}
L.~D. Landau and E.~M. Lifshitz, \emph{{Fluid Mechanics}}, 2nd ed., Course of Theoretical Physics, vol.~6, Pergamon Press, Oxford, 1987.

\bibitem{LapaHughes2014}
M.~F. Lapa and T.~L. Hughes, \emph{Swimming at low {Reynolds} number in fluids with odd ({Hall}) viscosity}, Phys. Rev. E \textbf{89} (2014), no.~4, 043019.

\bibitem{lixin2015}
J.~Li and Z.~Xin, \emph{Global existence of weak solutions to the barotropic compressible {Navier--Stokes} flows with degenerate viscosities}, preprint, \href{https://arxiv.org/abs/1504.06826}{\texttt{arXiv:1504.06826}} (2015).

\bibitem{Lions1996incomp}
P.-L. Lions, \emph{{Mathematical Topics in Fluid Mechanics, Volume 1: Incompressible Models}}, Oxford Lecture Series in Mathematics and Its Applications, vol.~1, Clarendon Press, Oxford, 1996.

\bibitem{lions1998comp}
P.-L. Lions, \emph{{Mathematical Topics in Fluid Mechanics, Volume 2: Compressible Models}}, Oxford Lecture Series in Mathematics and Its Applications, vol.~10, Oxford University Press, 1998.

\bibitem{polyatomic}
F.~R.~W. McCourt, J.~J.~M. Beenakker, W.~E. Kohler, and I.~Kuscer, \emph{{Nonequilibrium Phenomena in Polyatomic Gases: Cross Sections, Scattering, and Rarefied Gases}}, Oxford University Press, 1991.

\bibitem{mv2007}
A.~Mellet and A.~Vasseur, \emph{On the barotropic compressible {Navier--Stokes} equations}, Comm. Partial Differential Equations \textbf{32} (2007), no.~3, 431--452.

\bibitem{mv2008}
A.~Mellet and A.~Vasseur, \emph{Existence and uniqueness of global strong solutions for one-dimensional compressible {Navier--Stokes} equations}, SIAM J. Math. Anal. \textbf{39} (2008), no.~4, 1344--1365.

\bibitem{hamiltonian}
G.~M. Monteiro, A.~G. Abanov, and S.~Ganeshan, \emph{Hamiltonian structure of {2D} fluid dynamics with broken parity}, SciPost Phys. \textbf{14} (2023), 103.

\bibitem{gasrotation}
Y.~Nakagawa, \emph{The kinetic theory of gases for the rotating system.}, J. Phys. Earth \textbf{4} (1956), no.~3, 105--111.

\bibitem{superfluid}
N.~Read and E.~H. Rezayi, \emph{{Hall} viscosity, orbital spin, and geometry: Paired superfluids and quantum {Hall} systems}, Phys. Rev. B \textbf{84} (2011), 085316.

\bibitem{Simon1987}
J.~Simon, \emph{Compact sets in the space {$\lp(0,T;B)$}}, Ann. Mat. Pura Appl. \textbf{146} (1987), 65--96.

\bibitem{plasma}
L.~C. Steinhauer, \emph{Review of field-reversed configurations}, Phys. Plasmas \textbf{18} (2011), no.~7, 070501.

\bibitem{vaigant1995}
V.~A. Vaigant and A.~V. Kazhikhov, \emph{On existence of global solutions to the two-dimensional {Navier--Stokes} equations for a compressible viscous fluid}, Sib. Math. J. \textbf{36} (1995), 1108--1141.

\bibitem{vy2016existence}
A.~F. Vasseur and C.~Yu, \emph{Existence of global weak solutions for {3D degenerate compressible Navier--Stokes} equations}, Invent. Math. \textbf{206} (2016), 935--974.

\bibitem{vortex}
P.~Wiegmann and A.~G. Abanov, \emph{Anomalous hydrodynamics of two-dimensional vortex fluids}, Phys. Rev. Lett. \textbf{113} (2014), no.~3, 034501.

\bibitem{Zhang2015_JDE}
J.~Zhang, \emph{Global well-posedness for the incompressible {Navier--Stokes} equations with density-dependent viscosity coefficient}, J. Differential Equations \textbf{259} (2015), no.~5, 1722--1742.

\bibitem{zimmermann}
R.~Zimmermann, \emph{Weak solutions of the two-dimensional incompressible inhomogeneous {Navier--Stokes} equations in the presence of variable odd viscosity}, Anal. Theory Appl. (2026).

\end{thebibliography}
\end{document}